\documentclass[12pt, reqno]{amsart}

\usepackage{verbatim}
\usepackage{amssymb, amsmath, amsthm}
\usepackage{hyperref}
\usepackage[alphabetic,lite]{amsrefs}
\usepackage{amscd}   
\usepackage{fullpage}
\usepackage[all]{xy} 
\usepackage{pdflscape}
\DeclareFontEncoding{OT2}{}{} 

\usepackage[usenames, dvipsnames]{color}

\usepackage{graphicx}
\usepackage{tikz}
\usetikzlibrary{matrix,arrows,decorations.pathmorphing}

\usepackage{tikz-cd}
\usepackage{multicol}

\newtheorem{lemma}{Lemma}[section]
\newtheorem{theorem}[lemma]{Theorem}
\newtheorem{proposition}[lemma]{Proposition}
\newtheorem{prop}[lemma]{Proposition}
\newtheorem{cor}[lemma]{Corollary}

\newtheorem{claim}[lemma]{Claim}
\newtheorem{claim*}{Claim}

\newtheorem{example}[lemma]{Example}

\newtheorem*{theorem*}{Theorem}

\theoremstyle{definition}
\newtheorem{remark}[lemma]{Remark}

\newtheorem{rmk}[lemma]{Remark}

\newcommand{\A}{{\mathbb A}}
\newcommand{\G}{{\mathbb G}}

\newcommand{\Q}{{\mathbb Q}}
\newcommand{\R}{{\mathbb R}}
\newcommand{\Z}{{\mathbb Z}}

\newcommand{\Xbar}{{\overline{X}}}

\newcommand{\kbar}{{\overline{k}}}

\newcommand{\ksep}{{k^{\operatorname{sep}}}}

\newcommand{\Xsep}{{X^{\operatorname{sep}}}}

\newcommand{\kk}{{\mathbf k}}

\newcommand{\calA}{{\mathcal A}}

\newcommand{\calI}{{\mathcal I}}

\newcommand{\calO}{{\mathcal O}}

\newcommand{\calX}{{\mathcal X}}

\newcommand{\calZ}{{\mathcal Z}}

\usepackage{mathrsfs}

\DeclareMathOperator{\lcm}{lcm}

\DeclareMathOperator{\inv}{inv}

\DeclareMathOperator{\im}{im}

\DeclareMathOperator{\Hom}{Hom}

\DeclareMathOperator{\Gal}{Gal}
\DeclareMathOperator{\Ind}{Ind}

\DeclareMathOperator{\Res}{Res}

\DeclareMathOperator{\Br}{Br}

\DeclareMathOperator{\divv}{div}
\DeclareMathOperator{\ord}{ord}

\DeclareMathOperator{\Div}{Div}
\DeclareMathOperator{\Pic}{Pic}

\DeclareMathOperator{\Spec}{Spec}

\DeclareMathOperator{\ev}{ev}

\DeclareMathOperator{\et}{\textrm{\normalfont \'et}}

\DeclareMathOperator{\res}{res}
\DeclareMathOperator{\Princ}{Princ}

\DeclareMathOperator{\id}{id}

\newcommand{\isom}{\cong}

\numberwithin{equation}{section}
\numberwithin{table}{section}

\newcommand{\Btilde}{{\widetilde{B}}}

\newcommand{\Gtilde}{{\widetilde{G}}}

\newcommand{\Ubar}{\overline{U}}

\newcommand{\DivXminusU}{\Div_{\Xbar \setminus \Ubar}}
\newcommand{\gammatilde}{\widetilde{\gamma}}
\newcommand{\deltatilde}{\widetilde{\delta}}

\usepackage[shortlabels]{enumitem}

\DeclareMathOperator{\dd}{d}
\DeclareMathOperator{\D}{D}
\DeclareMathOperator{\Tot}{Tot}

\newcommand{\betatilde}{\widetilde{\beta}}
\newcommand{\kcycl}{k_\textrm{cycl}}
\newcommand{\varphitilde}{\widetilde{\varphi}}
\newcommand{\magma}{\textsc{magma}}
\usepackage{multicol}

\title{Obstructions to integral points on affine Ch\^atelet surfaces}
\author{Jennifer Berg}
\address{ Department of Mathematics MS 136, Rice University, Houston, TX 77005, USA}
\email{jb93@rice.edu}
\urladdr{http://math.rice.edu/\~{}jb93}

\begin{document}
\begin{abstract}
	We consider the Brauer-Manin obstruction to the existence of integral points on affine surfaces defined by $x^2 - ay^2 = P(t)$ over a number field. We enumerate the possibilities for the Brauer groups of certain families of such surfaces, and show that in contrast to their smooth compactifications, the Brauer groups of these affine varieties need not be generated by cyclic (e.g. quaternion) algebras. Concrete examples are given of such surfaces over $\Q$ which have solutions in $\Z_{p}$ for all $p$ and solutions in $\Q$, but for which the failure of the integral Hasse principle cannot be explained by a Brauer-Manin obstruction.

	The methods of this paper build off of the ideas of \cite{DSW15}, \cite{KT08}, \cite{Preu}: we study $\Br X/\Br_{0} X$ of affine Ch\^atelet surfaces with a focus on explicitly representing, by
	central simple algebras over $\kk(X)$, a finite set of classes of $\Br(X)$ which generate the quotient. Notably, we develop techniques for constructing explicit representatives of non-cyclic Brauer classes on affine surfaces, as well as provide an effective algorithm for the computation of local invariants of these Brauer classes via effective lifting of cocycles in Galois cohomology.
\end{abstract}
\maketitle

\section{Introduction}
Given a smooth, geometrically integral variety $X$ over $\Q$, one may ask whether $X$ has a rational point or integral point, i.e. whether $X(\Q) \ne \emptyset$ or $\calX(\Z) \ne \emptyset$ (for $\calX$ a separated scheme of finite type over $\Z$ with generic fiber $X$).  Since $\Q$ embeds into each of its completions, a necessary condition for $X(\Q)$ (resp. $\calX(\Z)$) to be nonempty is that $X$ is locally soluble, that is $X(\Q_p) \ne \emptyset$ for all primes $p$ of $\Q$ (resp. $\calX(\Z_p) \ne \emptyset$). This condition is often not sufficient; there exist varieties $X$ which are everywhere locally soluble, but $X(\Q) = \emptyset$ (resp. $\calX(\Z) = \emptyset$). These varieties are said to \textbf{fail the (integral) Hasse principle}.

In 1970, Manin \cite{Man71} explained many failures of the Hasse principle for rational points via use of the Brauer group and class field theory. In particular, he defined a subset $X(\A_\Q)^{\Br}$ of $X(\A_\Q)$, known today as the \textbf{Brauer-Manin set}, with the property that 
\[ X(\Q) \subseteq  X(\A_\Q)^{\Br} \subseteq X(\A_\Q ).\]
Thus, an empty Brauer-Manin set can be viewed as an obstruction to the existence of rational points.

In general, the study of integral points is more difficult than that of rational points on the same variety, although when $X$ is proper, the two notions coincide. Thus, when studying obstructions to integral points, it is sensible to restrict attention to non-proper varieties. In 2009, Colliot-Th\'el\`ene and Xu \cite{CTXO9} made the significant observation that the Brauer group and Brauer-Manin set are relevant to the study of integral points as well.

When interested in solutions over the integers, one must consider another issue aside from local solubility and the Brauer-Manin obstruction. Consider, for example, the surface $\calX_{22} \subset \A_\Z^3$  with defining equation $x^2 + y^2 + t^4 = 22$, which is insoluble in $\Z^3$, although there exist both rational solutions as well as solutions in $\Z_p$ for each prime $p$. The simplest argument to show insolubilty over $\Z$ is to notice that the equation defines a compact submanifold in $\R^3$, thereby restricting the number of integral points to be finite. This observation allows for the relatively simple construction of examples of separates schemes of finite type over $\Z$ which have neither an integral point nor a Brauer-Manin obstruction, in stark contrast with the situation for rational points. For example, in \cite{CTW12} it is proved that $\calX \subset \A_\Z^3$ defined by the equation $2x^2 + 3y^2 + 4z^2 =1$, which is clearly insoluble over $\Z$, has no Brauer-Manin obstruction to integral points. Again, this example suffers from the fact that $X(\R)$ is compact. Further examples in which archimedean phenomena act in opposition to the Brauer-Manin obstruction can be found in \cites{DW16}, \cites{JS16}, \cites{Har17}, \cites{OW16}. In \S \ref{sec: concrete examples} we prove that the insolubility of $\calX_{22}$ over $\Z$, as well as other examples of a similar nature, cannot be explained by a Brauer-Manin obstruction. Moreover, we prove that this remains true when we consider $\Z[\frac{1}{2}]$ points, demonstrating that it is not enough to require unboundedness at one place.

In order to study the Brauer-Manin obstruction to the existence of integral points in \S\ref{sec: concrete examples}, we must start by considering a larger class of varieties. To that end, let $L/k$ be an extension of number fields of degree $d$, and choose a basis $\omega_1, \dots, \omega_d$ for $L/k$. Consider the affine $k$-variety $X \subset \A^{d+1}_k$ defined by $ N_{L/k}(\vec{z}) = P(t) $
where $N_{L/k}(\vec{z}) = N_{L/k}(z_1 \omega_1 + \dots + z_d \omega_d)$ is the norm form for the extension $L/k$, and $P(t)$ is a separable polynomial over $k$. These varieties will be referred to as norm form varieties. In particular, we will be interested in the case when $\dim X = 2$; such $k$-varieties are called \textbf{(generalized) affine Ch\^atelet surfaces}, and have the simpler form
\begin{equation} \label{eqn: Xoverk}
	x^2 - ay^2 = c P(t)
\end{equation}
with $c \in k^\times$, $a \in k^\times$ not a square, and $P(t) \in k[t]$ a monic separable polynomial of degree $n$. Since $P(t)$ is separable, $X$ is a smooth geometrically integral surface over $k$.

The study of the Hasse principle for rational points and weak approximation is well understood on the smooth proper compactifications of Ch\^atelet surfaces. In the monumental two part paper \cites{CTSSD87a,CTSSD87b} Colliot-Th\'el\`ene, Sansuc, and Swinnerton-Dyer prove that the only obstruction to the Hasse principle and weak approximation on classical Ch\^atelet surfaces is the Brauer-Manin obstruction.

Thus, it is natural to wonder whether an analogous result holds for affine Ch\^atelet surfaces; in \cite{CTH13}, Colliot-Th\'el\`ene and Harari ask for the integral Hasse principle and strong approximation for such surfaces with $\deg P(t) \ge 3$. One can generalize further to norm form varieties; when $[L:k] \ge 4$ and $\deg P(t) \ge 2$, these varieties share the same geometric properties with affine Ch\^atelet surfaces that were incompatible with the known techniques for answering this question at the time. Using descent methods, Derenthal and Wei \cite{DW16} give the first strong approximation (with Brauer-Manin obstruction) results for $[K:k] = 4$ and $\deg P(t) = 2$. In the case of affine Ch\^atelet surfaces, Gundlach \cite{Gun13} gives strong approximation results under the assumption of Schinzel's hypothesis for surfaces defined by $x^2 + y^2 + z^k = m$ with $k \ge 3$, odd.

In general, the first step to understanding Brauer-Manin obstructions is to determine the isomorphism type of the Brauer group $\Br X/\Br_0 X$. In all of the aforementioned cases, we note that the Brauer groups of these varieties modulo constant algebras are generated by two torsion, e.g. quaternion algebras. We prove that this need not be the case for affine Ch\^atelet surfaces.

\begin{theorem} \label{thm: short version degree 4} Let $P(t)$ be a quartic polynomial over a number field $k$ that is irreducible and has splitting field $K$. Let $X \subset \A_k^3$ be an affine Ch\^atelet surface as in (\ref{eqn: Xoverk}) and let $L = k(\sqrt{a})$. Then the Brauer group of $X$ modulo constant algebras is either trivial or generated by quaternion algebras, except when $L$ is contained in $K$, $\Gal(K/k) \isom D_4$ and $\Gal(K/L) \isom \Z/4\Z$ in which case it is generated by a non-cyclic algebra of order 4.

\end{theorem}

In \S\ref{sec: BrX_isomtype} we prove a more detailed version of Theorem \ref{thm: short version degree 4}, and an give an explicit description of Brauer classes in the nontrivial cases in \S\ref{subsec: H2BrU}. For higher degree polynomials, the calculations become increasingly intricate, and a complete classification would be quite involved. However, we prove (see \S \ref{subsec: H2BrU}, Theorem \ref{thm: BrX_explicitrep} for an explicit version):

\begin{theorem} \label{thm: short version general dihedral} Let $P(t)$ be a polynomial of degree $n$ over a number field $k$ that is irreducible and has splitting field $K$. Let $X \subset \A_k^3$ be an affine Ch\^atelet surface as in (\ref{eqn: Xoverk}) and let $L = k(\sqrt{a})$. If $\Gal(K/k) \isom D_n$, then the Brauer group of $X$ modulo constant algebras is non-trivial and generated by a non-cyclic algebra of order $n$ if $L$ is contained in $K$ and $\Gal(K/L) \isom \Z/n\Z$.
\end{theorem}

In these cases, we then construct explicit representatives, in terms of central simple algebras over the function field $\kk(X)$ of $X$, of the finite set of classes of $\Br X$ which generate modulo constant algebras. We prove,

\begin{theorem} There exists an effective algorithm to compute the Brauer-Manin set for affine Ch\^atelet surfaces defined by (\ref{eqn: Xoverk}) with $P(t)$ an irreducible quartic polynomial, or dihedral of arbitrary degree satisfying the conditions of Theorem \ref{thm: short version general dihedral}.
\end{theorem}

We then use this algorithm to prove

\begin{prop} 
	Let $S = \{2, \infty\}$. The Brauer-Manin obstruction is insufficient to explain failures of the S-integral Hasse principle for the $\Z$-schemes $\calX_m: x^2 + y^2 + t^4 = m$, with $m = 22,43,67,70,78,93,177, \dots$
\end{prop}

\subsection{Strategy and Outline}
In \S \ref{sec: enumeration_of_BrX}, we prove results about the geometry of affine Ch\^atelet surfaces. In particular, we give a concrete description of the geometric Picard group and the Galois action on the divisors which generate this group over a separable closure of the ground field. This allows us to enumerate isomorphism types of the (algebraic) Brauer group at the end of the section.

In \S \ref{section: explicit_classes} we detail an approach for computing representatives of (non-cyclic) Brauer classes explicitly via cocycles in Galois cohomology. This relies on the use of non-standard resolutions to efficiently compute cohomology. In \S \ref{subsec: efficres} with this efficient resolution, we are able to write down 1-cocycles whose classes generate the group $H^1(\Gal(K/k), \Pic X_K)$. In \S \ref{subsec: H2BrU}, we relate these 1-cocycles to 2-cocycles in $H^2(\Gal(K/k), \calO(U_K)^\times)$, where $U$ is the open subscheme defined by $P(t) \ne 0$. These 2-cocycles represent the restrictions to $\Br U$ of classes in $\Br X$ which generate. This enables us to prove Theorem \ref{thm: BrX_explicitrep}, in which we construct explicit representatives for the generators of the Brauer groups computed in the previous section.

In \S \ref{sec: cocycle_lifting}, with explicit representatives of Brauer classes in hand, we give an effective algorithm for computing the Brauer-Manin set, $\calX(\A_{k,S})^{\Br}$. In particular, we provide an approach for effective lifting of 2-cocycles in Galois cohomology, based on the ideas of \cite{Preu}.

Finally, in \S \ref{sec: concrete examples} we employ the tools from the previous sections to compute the Brauer-Manin set for a collection of affine Ch\^atelet surfaces. We prove that the Brauer-Manin obstruction is not sufficient to explain all failures of the integral Hasse principle for these surfaces.

\subsection{Terminology}
For a field $k$ of characteristic 0, let $\kbar$ be a fixed algebraic closure of $k$, and $\ksep$ a fixed separable closure within $\kbar$. Let $G_k$ denote the absolute Galois group $\Gal(\ksep/k)$. A $k$-variety is a separated scheme of finite type over $k$. For a $k$-scheme $X$ and a field extension $K/k$, we set $X_K := X \times_{\Spec k} \Spec K$, and let $\Xbar = X_{\kbar}$.  Let $K[X] := H^0(X_K, \calO_{X_K})$ denote the ring of global functions on $X_K$, and let $K[X]^\ast$ denote the group of units in that ring. On occasion we will also use the notation $\calO(X)^\times$ to denote the units.

Let $\Br(k)$ denote the Brauer group of $k$. For a scheme $X$ over $k$, we have the following filtration of the Brauer group \cite{Gro68}:
\[ \Br_0 X := \im(\Br k \to \Br X) \subseteq \Br_1 X := \ker(\Br X \to \Br \Xsep) \subseteq \Br X := H^2_{\et}(X, \G_m). \]
The elements in $\Br_0 X$ are said to be constant, and elements in $\Br_1 X$ are said to be algebraic.

The Picard group of $X$ is $\Pic X := \Div X/ \Princ X$, where $\Div X$ denotes the group of Weil divisors on $X$ and $\Princ X$ denotes the group of principal divisors on $X$. For a divisor $D \in \Div X$, we write $[D]$ for its equivalence class in $\Pic X$. For any subvariety $U \subset X$, let $\Div_{X \setminus U}(X)$ denote the classes of divisors supported in $X \setminus U$.

Now let $k$ be a number field, $\calO_k$ its ring of integers, and let $\Omega_k$ denote its set of places. For $v \in \Omega_k$ finite, let $k_v$ be the completion of $k$ at the place $v$, and for $v$ archimedean let $\calO_v$ be the ring of integers in $k_v$. For each place $v$, class field theory gives an embedding $\inv_v \colon \Br k_v \to \Q/\Z$.

The adele ring with its usual adelic topology is denoted by $\A_k$. For $S$ a finite set of places of $k$ containing all archimedean places, let $\calO_S$ be the ring of $S$-integers of $k$.
 Let $\calX$ be a separated scheme of finite time over $\calO_S$. The set $\calX(\calO_S)$ denotes the set of $S$-integral points of $X$. Let $X:= \calX \times_{\calO_S} k$.

Given a scheme $X$ and a point $x_v \in X(k_v)$, there exists a restriction homomorphism $x_v^\ast \colon \Br X \to \Br k_v \isom \Q/\Z$, where the latter isomorphism is $\inv_v$, the local invariant map. For a fixed class $\alpha \in \Br X$, the local evaluation map is 
\[ \ev_{\alpha,v} \colon X(k_v) \to \Q/\Z, \hspace{1 em} x_v \mapsto \alpha(x_v):=x_v^\ast(\alpha). \]
Moreover, there exists a finite set of places $S_{\alpha, \calX}$ with $S \subset S_{\alpha, \calX}$ such that for any $v \not \in S_{\alpha, \calX}$ and $x_v \in \calX(\calO_v)$, we have $\alpha(x_v) = 0$.

As above, let $S$ be a finite set of places containing all archimedean places, and let $X/k$ be a non-proper variety. There exists a pairing
\[ \Big [\prod_{v \in S} X(k_v) \times \prod_{v \not \in S} \calX(\calO_v) \Big] \times \Br X \to \Q/\Z. \]
The Brauer-Manin set in this context is by definition the left kernel of this pairing, i.e.
\[\left( \prod_{v \in S} X(k_v) \times \prod_{v \not \in S} \calX(\calO_v) \right)^{\Br X} = \left \{ x_v \in \prod_{v \in S} X(k_v) \times \prod_{v \not \in S} \calX(\calO_v) \Big\rvert  \sum_v \inv_v \alpha(x_v) = 0 \right \}.
\]
We have the inclusions: 
\[ \calX(\calO_S) \subset \left( \prod_{v \in S} X(k_v) \times \prod_{v \not \in S} \calX(\calO_v) \right)^{\Br X} \subset  \prod_{v \in S} X(k_v) \times \prod_{v \not \in S} \calX(\calO_v) .\]

If the product $\prod_{v \in S} X(k_v) \times \prod_{v \not \in S} \calX(\calO_v)$ is non-empty, but the Brauer-Manin set is empty, we say there is a \textbf{ Brauer-Manin obstruction} to the existence of an $S$-integral point on $\calX$. The Brauer-Manin set depends only on the quotient $\Br X/ \Br_0 X$.

\subsection*{Acknowledgements} I thank Felipe Voloch for suggesting this problem and many valuable conversations. I am indebted to Bianca Viray and Anthony V\'arilly-Alvarado for countless helpful discussions and comments on this work. I also thank Andrew Kresch for several correspondences. All computations were done using $\magma$ \cite{MAGMA}. The relevant scripts are available from the author.

\section{Enumeration of Brauer groups} \label{sec: enumeration_of_BrX}
In this section, we study the geometry of affine Ch\^atelet surfaces. In order to compute the Brauer group of such surfaces, we first need a concrete description of the divisors generating the Picard group and the Galois action on those divisors. With this description in hand, we then enumerate the isomorphism types of the Brauer groups of families of affine Ch\^atelet surfaces.

\subsection{Geometry of $\Xbar$}
	Consider the surjective morphism
	\[ \pi \colon X \to \A^1_k, \hspace{1 em} (x,y,t) \mapsto t.\]
	Let $U_0 \subset \A_1^k$ be defined by the condition $P(t) \ne 0$, and let $U = \pi^{-1}(U_0) \subset X$. As $U_0$ is an open subscheme in $\A_k^1$, $\Pic(U_0) = 0$. Thus, $\Pic \Ubar = 0$, since 
	\[ \Ubar = U \times_k \kbar \isom \overline{U_0} \times \G_{m, \kbar}. \]

	Over $\kbar$ it will be useful to write 
	\[ \Xbar \colon u_1 u_2 = c(t - e_1)(t - e_2) \dots (t - e_n) \]
	with $u_1 = x - \sqrt{a}y$ and $u_2 = x + \sqrt{a} y$. There exist $n$
	reduced fibers of $\overline{\pi} \colon \Xbar \to \A_{\kbar}^1$, corresponding to the roots of $P(t)$, each consisting of two irreducible components. Thus, $\DivXminusU(\Xbar)$ is free of rank $2n$ with basis given by the divisors $D_{i,j}$, $i = 1,2$ and  $j = 1,\dots, n $ within defining equations
	\[ D_{i,j} : = \{t - e_i = 0, u_j = 0\}. \]

	We observe first that $X$ is a generalized affine Ch\^atelet surface, then $\Br \Xbar = 0$.
	\begin{lemma} \cite[Lemma 4.2]{DW16} \label{lem: BrXBar} 
		Let $k$ be an algebraically closed field of characteristic zero. Let $
		P(t) = c(t- e_1) \dots (t - e_r) \in k[t].$ Let $Y \subset \A_k^{s+1}$ be the affine variety defined by $z_1 \dots z_s = P(t)$. Then $\Br Y = 0$.
	\end{lemma}

	Next, we show that the only invertible functions on $\Xbar$ are the constant functions. 

	\begin{lemma} \label{lem: fxnsonXbar}
		 Let $X$ be a generalized affine Ch\^atelet surface. Then $\kbar[X]^\times = \kbar^\times$.
	\end{lemma}

	\begin{proof} 
		Consider the morphism $\overline{\pi} \colon: \Xbar \to \A^1_{\kbar}$. The generic fiber of this morphism is isomorphic to $\G_{m, \kbar(t)}$, hence the regular functions on $\Xbar$ are generated by those of the form $f = u_1^{a_1} u_2^{a_2} g(t)$ with $g(t) \in k(t)$. If $g(t)$ has a zero or pole at $t_0 \not \in \{e_1, \dots, e_n\}$, then either $f$ or $f^{-1}$ will fail to be regular at $\pi^{-1}(t_0)$. Thus, $f$ must have the form 
		\[f = c' u_1^{a_1} u_2^{a_2} (t-e_1)^{b_1} \dots (t - e_n)^{b_n}\]
		for some $c' \in \kbar^\times$. Applying $\divv$ to $f$ yields, 
		\[ \divv(f) = \sum_{i = 1}^2 \sum_{j =1}^n (a_i + b_j) \D_{i,j}. \]
		Hence $f \in \kbar[X]^{\times}$ if and only if $a_1 = a_2 = -b_1 = \dots = -b_n$. However, the defining equation for $X$ forces $f$ to be constant. 
	\end{proof}

	Finally, we conclude this section with the observation that all Brauer classes on $X$ must in fact be algebraic. 

	\begin{lemma} \label{BrXAlg}
		Let $X$ be a generalized affine Ch\^atelet surface as in (\ref{eqn: Xoverk}). Then $\Br X = \Br_1 X$, hence 
		\[ \Br X/ \Br_0 X \isom H^1(G_k, \Pic \Xbar).\]
	\end{lemma}

	\begin{proof} 
		The first claim is an immediate corollary of Lemma \ref{lem: BrXBar}. Lemma \ref{lem: fxnsonXbar} shows that $\kbar[X]^\times = \kbar^\times$, hence we may apply the Hochshild-Serre spectral sequence, whose long exact sequence of low degree terms yields the desired isomorphism. 
	\end{proof}

\subsection{Galois Module Structure}
Let $K$ be the splitting field of the polynomial $P(t)$, with Galois group $\Gal(K/k) = G$. Consider the inflation restriction sequence 
	\[
		\begin{tikzcd}
			0 \to H^1(\Gal(K/k), \Pic(\Xbar)^{G_K}) \arrow{r}{\inf} & H^1(k, \Pic(\Xbar)) \arrow{r}{\res}& H^1(K, \Pic(\Xbar)).
		\end{tikzcd}
	\]

	Since $\Pic(\Xbar)$ is a torsion-free $G_K$-module with trivial Galois action, $H^1(K, \Pic(\Xbar)) = \Hom_{cts}(G_K, \Pic(\Xbar)) = 0$. Thus, we obtain an isomorphism
	\[
		H^1(\Gal(K/k), \Pic(X_K)) \cong H^1(k, \Pic(\Xbar))
	\]

	\begin{lemma} \label{lem: PicXbar} The following sequence of $G$-modules is exact:
	\begin{equation}\label{eq: GalKSES}
		\begin{tikzcd}
			0 \ar{r} & \calO(U_K)^\times/ K^\times \ar{r}{\divv} & \Div_{X_K \setminus U_K}(X_K) \ar{r} & \Pic(X_K) \ar{r} & 0
		\end{tikzcd}
	\end{equation}
	\end{lemma}

	\begin{proof}
		The sequence of $G_k$-modules
		\begin{equation}\label{eq: Gkexact}
			\begin{tikzcd}[column sep = 1.5 em]
				\kbar[X]^\times/\kbar^\times \arrow{r} & \kbar[U]^\times/\kbar^\times \arrow{r} & \DivXminusU(\Xbar) \arrow{r} & \Pic(\Xbar) \arrow{r} & \Pic(\Ubar)
			\end{tikzcd}
		\end{equation}
		is exact. Lemma \ref{lem: fxnsonXbar} shows that $\kbar[X]^\times/ \kbar^\times$ is trivial and as noted above, $\Pic(\Ubar) = 0$, hence we have
		\begin{equation}\label{eq: GkSES}
			\begin{tikzcd}
				0 \arrow{r} & \kbar[U]^\times/\kbar^\times \ar{r}{\divv} & \DivXminusU(\Xbar) \arrow{r} & \Pic(\Xbar) \arrow{r} & 0
			\end{tikzcd}
		\end{equation}
		In fact, (\ref{eq: GkSES}) holds over $K$, and since $\Pic X_K = \Pic \Xbar$, we obtain the exact sequence (\ref{eq: Gkexact}).
	\end{proof}

	We will thus make use of (\ref{eq: GalKSES}) to compute the structure of $\Pic \Xbar$ as a Galois module. The abelian group $K[U]^\times/K^\times$ is free of rank $n+1$, and is generated by the classes of the functions $t-e_1, \dots, t-e_n, u_1, u_2$ with the following relation, which arises from the defining equation for $X_K$:
	\begin{equation}\label{eqn: KUrelation}
		\sum_{i=1}^n [t-e_i] = [u_1] + [u_2]
	\end{equation}

	We consider the homomorphism $\divv: K[U]^\times/K^\times \to \Div_{X_K \setminus U_K}(X_K)$ sending a function to its divisor. Then, we have
	\vspace{-1 em}
	\begin{align*}
		&\divv(t- e_j) = D_{1,j} + D_{2,j} \\
		&\divv(u_i) = \displaystyle \sum_{j=1}^n D_{i,j}
	\end{align*}

	Thus, since the sequence (\ref{eq: GalKSES}) is exact, it follows that $\Pic(X_K)$ is a free $\Z$-module of rank $n-1$, generated by $[D_{i,j}]$ with the relations
	\begin{eqnarray*}
		&\sum_{j=1}^n [D_{1,j}] = 0 & { } \\
		&{[D_{1,j}]} + {[D_{2,j}]} = 0, & j=1,\dots, n.
	\end{eqnarray*}

	Finally, we fix the notation $k_t := k[t]/(P(t))$, and $L := k(\sqrt{a})$. Then the $G$-modules exhibit the structure
	\begin{eqnarray}
	\label{Div} \Div_{X_K \setminus U_K}(X_K) \isom \Z[k_t/k] \otimes \Z[L/k]
	\end{eqnarray}
	Moreover, we have the exact sequence 
	\begin{equation} \label{eqn: KUGalModule}
		\begin{tikzcd}
		0 \ar{r} & \Z \ar{r} & \Z[k_t/k] \oplus \Z[L/k] \ar{r} & K[U]^\times/ K^\times \ar{r} & 0
		\end{tikzcd}
	\end{equation}
	arising from the relation (\ref{eqn: KUrelation}).

Finally, we record some results from group cohomology.

\begin{lemma}[\cite{BrownKS}] \label{prop: indresstructure}
	Let $N$ be a $G$-module. Then, 
	$$\Ind_H^G \Res_H^G N \isom \Z[G/H] \otimes N,$$
	where $G$ acts diagonally on the tensor product.
	\end{lemma}

\begin{lemma} \label{claim: induced} Let $H$ be a normal subgroup of a finite group $G$ and let $H'$ be the complementary subgroup, so that $H \cap H' = \{1\}$. Then the module $M := \Z[G/H'] \otimes \Z[G/H]$ has trivial cohomology $H^i(G,M)$ for $i > 0$. 
	\end{lemma}

\subsection{Brauer Groups} \label{sec: BrX_isomtype}
We are now ready to state the extended version of Theorem \ref{thm: short version degree 4}.

\begin{theorem} \label{thm: extended brauer deg 4}
	Let $P(t)$ be a separable quartic polynomial over a number field $k$ that is irreducible and has splitting field $K$, with Galois group $G$. Let $X \subset \A_k^3$ be an affine Ch\^atelet surface with defining equation $x^2 - ay^2 = c P(t)$, with $a, c \in k^\times$, and $a$ not a square in $k$. Let $L = k(\sqrt{a})$, and let $k_t := k[t]/(P(t))$. If $L$ is contained in $K$ then the Brauer group of $X$ modulo constant algebras is
	\[ \Br X/ \Br k \isom \begin{cases} \Z/4\Z, & \textup{ if } G \isom D_4 \textup{ and } \Gal(K/L) \isom \Z/4\Z \\
	\Z/2\Z, & \textup{ if } G \isom (\Z/2\Z)^2 \\
	& \textup{ or } G \isom D_4 \textup{ and } L \not \subset k_t \\
	0, & \textup{ otherwise } 
	\end{cases} 
	\]
	\end{theorem}

\begin{rmk}
	As the techniques of this section will illustrate, the isomorphism types of Brauer groups of affine Ch\^atelet surface with other Galois actions can be computed in a similar manner. Moreover, we observe that if $P(t) = \prod_{i=1}^r P_i(t)$ is not irreducible, then $\Br X/ \Br_0 X$ contains (and is generated by) the nontrivial quaternion algebras $\big(-a, P_i(t)\big) \in \Br \kk(X)$. See \cite{Gun13}, \cite{DW16}, for example. Moreover similar techniques can be used to check that in the theorem above, when $L$ is not contained in $K$, $\Br X/ \Br k$ is $\Z/2\Z$ if $G$ is abelian, and otherwise is trivial.
\end{rmk}

\begin{proof}
	We observe that if $G \cong A_4$, no quadratic subfield of $K$ exists ($A_4$ has no subgroups of index 2), hence we may exclude this case. We will individually treat the following remaining cases arising from the possibilities for the Galois group $G$ of an irreducible quartic polynomial:
	\begin{enumerate}
		\item $G \isom \Z/4\Z$ 
		\item $G \isom \Z/2\Z \times \Z/2\Z$
		\item $G \isom D_4$ and $L \subset k_t$
		\item $G \isom D_4$, $L \not \subset k_t$, and $\Gal(K/L) \isom \Z/2\Z \times \Z/2\Z$
		\item $G \isom D_4$ and $\Gal(K/L) \isom \Z/4\Z$
		\item $G \isom S_4$
	\end{enumerate}

	\noindent {\bf Case (1).} When $G \isom \Z/4\Z$, $k_t \isom K$, and we have that $\Div_{X_K \setminus U_K}(X_K) \isom \Z[K/k] \otimes \Z[L/k]$, hence $H^i(G,\Div_{X_K \setminus U_K}(X_K)) = 0$ for all $i > 0$. Thus by the exact sequence (\ref{eq: GalKSES}), \[ H^1(G, \Pic X_K) \isom H^2(G, K[U]^\times/ K^\times)\]
	From the long exact sequence in cohomology corresponding to (\ref{eqn: KUGalModule}), we have 
\begin{equation} \small \label{LES: Brauer}
		\begin{tikzcd}[row sep = .5 em, column sep = 1 em]
		H^2(K/k,\Z) \ar{r} & H^2(K/k, \Z[K/k] \oplus \Z[L/k]) \ar{r} \ar[d, phantom,  ""{coordinate, name=Z}] & H^2(K/k, K[U]^\times/ K^\times) \arrow[dll, rounded corners, to path={ -- ([xshift=2ex]\tikztostart.east)
		|- (Z) [near end]\tikztonodes
		-| ([xshift=-2ex]\tikztotarget.west)
		-- (\tikztotarget)}]\\
		H^3(K/k,\Z) \ar{r} & H^3(K/k, \Z[K/k] \oplus \Z[L/k]) 
		\end{tikzcd}
	\end{equation}

	Then $H^2(K/k, \Z[K/k]) = 0$, since this is an induced module, and by Shapiro's lemma we have $H^2(K/k, \Z[L/k]) \isom H^2(K/L, \Z)$. Thus, the first map is the natural surjective map $H^2(K/k,\Z) \to H^2(K/L, \Z)$. Similarly, we have that $H^3(K/k, \Z[K/k]) = 0$, and since $K/L$ is cyclic, $H^3(K/L, \Z) =0$. Thus, the third map in the exact sequence above is an isomorphism. But since $K/k$ is cyclic, this gives that $H^2(K/k,K[U]^\times/K^\times) \isom H^3(K/k,\Z) = 0$. Thus $\Br X = \Br_0 X$. \bigskip

	\noindent {\bf Case (2).} As in Case 1, when $G \isom \Z/2\Z \times \Z/2\Z$, we find $\Div_{X_K \setminus U_K}(X_K) \isom \Z[K/k] \otimes \Z[L/k]$, so that $H^1(K/k, \Pic(X_K)) \isom H^2(K/k, K[U]^\times/K^\times)$. The sequence (\ref{LES: Brauer}) holds in this setting as well, and again we observe that the first map is the natural surjection $H^2(K/k, \Z) \to H^2(K/L, \Z)$, and that the third map is an isomorphism. Thus, $H^2(K/k,K[U]^\times/K^\times) \isom H^3(K/k,\Z) = \Z/2\Z$. Thus, $\Br X/\Br k \isom \Z/2\Z$. \bigskip

	For the remaining isomorphism types, the $G$-module structure (\ref{Div}) gives $\Div_{X_K \setminus U_K}(X_K) \isom \Z[k_t/k] \otimes \Z[L/k]$. However, we observe that now $k_t \not \isom K$, thus a priori we may only conclude that $ H^1(K/k, \Div_{X_K \setminus U_K}(X_K)) =0,$ since this is a permutation module. Additionally, the long exact sequence in cohomology of (\ref{eqn: KUGalModule})
	simplifies to
	\begin{equation} \small \label{LES: simplifiedS4}
		\begin{tikzcd}[column sep = 1 em, row sep = .5 em]
		H^2(K/k,\Z) \ar{r} & H^2(K/k_t, \Z) \oplus H^2(K/L, \Z) \ar{r} \ar[d, phantom,  ""{coordinate, name=Z}] & H^2(K/k, K[U]^\times/ K^\times) \arrow[dll, rounded corners, to path={ -- ([xshift=2ex]\tikztostart.east)
		|- (Z) [near end]\tikztonodes
		-| ([xshift=-2ex]\tikztotarget.west)
		-- (\tikztotarget)}]\\
		H^3(K/k,\Z) \ar{r} & H^3(K/k_t, \Z) \oplus H^3(K/L, \Z)
		\end{tikzcd}
	\end{equation}

	\noindent {\bf Case (3).} When $L$ is contained in $k_t$, $\Div_{X_K \setminus U_K}(X_K)$ is not an induced module. However, Prop \ref{prop: indresstructure} can be used to compute the cohomology, and we find that $H^2(G,\Div_{X_K \setminus U_K}(X_K)) \isom \Z/2\Z \oplus \Z/2\Z$. Additionally, we compute that in that case $H^2(G, K[U]^\times/K^\times) \isom \Z/2\Z \oplus \Z/2\Z$, and thus, $H^1(G, \Pic X_K) = 0$. \bigskip

	\noindent {\bf Case (4).} When $L \not \subset k_t$ and $\Gal(K/L) \isom \Z/2\Z \times \Z/2\Z$, sequence (\ref{LES: simplifiedS4}) yields
	\[
		\begin{tikzcd}[column sep = 1 em, row sep = .75 em]
		\Z/2\Z \oplus \Z/2\Z \ar{r} & \Z/2\Z \oplus (\Z/2\Z \oplus \Z/2\Z) \ar{r} &
		H^2(G,K[U]^\times/K^\times) \ar{r} & \Z/2\Z \ar{r} & \Z/2\Z 
		\end{tikzcd}
	\]
	and we obtain $H^2(G,K[U]^\times/K^\times) \isom \Z/2\Z$. \bigskip

	\noindent {\bf Case (5).} When $\Gal(K/L) \isom \Z/4\Z$, we claim that $\Div_{X_K \setminus U_K}(X_K)$ is an induced $G$-module, hence $H^i(G, \Div_{X_K \setminus U_K}(X_K)) =0 \textup{ for } i>0.$ To see this, recall that $H$ is the (cyclic) normal subgroup generated by $\sigma$ and let $H'$ be the complementary subgroup, so that $H \cap H' = \{1 \}$. Then $\Div_{X_K \setminus U_K}(X_K) \isom \Z[G/H'] \otimes \Z[G/H]$, and the claim follows immediately from Lemma \ref{claim: induced}.
	Since (\ref{eq: GalKSES}) is exact, this shows that $H^1(G, \Pic X_K) \isom H^2(G, K[U]^\times/ K^\times)$. Thus sequence (\ref{LES: simplifiedS4}) yields
	\[ \begin{tikzcd}[row sep = 0 em, column sep = 1.25 em]
		\Z/2\Z \oplus \Z/2\Z \ar{r} & \Z/2\Z \oplus \Z/4\Z \ar{r} & H^2(K/k, K[U]^\times/ K^\times) 
		\ar{r} & \Z/2\Z \ar{r} & 0
		\end{tikzcd}
	\]
	With the aid of \magma \,we find $H^2(K/k, K[U]^\times/ K^\times) \isom \Z/4\Z$. (See $\S3$).

	\smallskip

	\noindent {\bf Case (6).}  Finally, suppose that $G \isom S_4$. Then, sequence (\ref{LES: simplifiedS4}) gives
	\[
	\begin{tikzcd}[row sep = .75 em] \small
		H^2(S_4, \Z) \ar{r} & H^2(S_3, \Z)  \oplus H^2(A_4, \Z) \ar{r} &
		H^2(S_4,\Z) \ar{r} & H^2(S_3, \Z)  \oplus H^2(A_4, \Z) 
		\end{tikzcd}
	\]

	If we consider the inflation-restriction sequence for $S_4$ and the normal subgroup $\Z/2\Z \times \Z/2\Z$
	we find that $H^2(S_3, \Z) \isom H^2(S_4, \Z)$, hence the first map above is injective. This yields $H^2(G, K[U]^\times/K^\times) \isom \Z/3\Z$. Thus we have
	\[ \begin{tikzcd} 0 \ar{r} & H^1(G,\Pic(X_K)) \ar{r} & \Z/3\Z \ar{r} & H^2(G, \Div_{X_K \setminus U_K}(X_K))
	\end{tikzcd}
	\]

	Lemma \ref{prop: indresstructure} provides a description of $\Div_{X_K\setminus U_K}(X_K)$, and allows us to compute that $H^2(G, \Div_{X_K \setminus U_K}(X_K)) \isom \Z/3\Z$, and the third map above is an isomorphism. Hence we find 
	$H^1(G, \Pic X_K) = 0$.

	We note that these calculations can be carried out with the aid of \magma\, as well.
\end{proof}

\noindent Finally, we prove Theorem \ref{thm: short version general dihedral}.
\begin{proof}
		We first recall the low degree integral cohomology groups of $D_{2n}$. When $n$ is odd, these groups can be computed from the Lyndon-Hochschild-Serre spectral sequence:
		\[
		{H^i(D_{2n}, \Z) =
			\begin{cases}
			\Z, & \textup{ if } i = 0 \\
			\Z/2\Z, & \textup{ if } i \equiv 2 \pmod{4} \\
			\Z/2n \Z & \textup{ if } i \equiv 0 \pmod{4} \\
			0, & \textup{ if } i \textup{ is odd} 
			\end{cases}}
		\]

		and when $n$ is even (\cite{Handel}):
		\[ 
		{H^i(D_{2n}, \Z) =
			\begin{cases}
			\Z, & \textup{ if } i = 0 \\
			(\Z/2\Z)^{i-1/2}, &  \textup{ if } i \textup{ is odd} \\
			(\Z/2\Z)^{i+2/2}, & \textup{ if } i \equiv 2 \pmod{4} \\
			\end{cases}}
		\]

	We consider first the case of odd $n$. The proof of claim (\ref{claim: induced}) holds in this context. That is, $\Div_{X_K \setminus U_K}(X_K)$ is an induced module, hence $H^1(G, \Pic X_K) \isom H^2(G, K[U]^\times/ K^\times)$. 
	When $L \subset K$, the sequence (\ref{LES: simplifiedS4}) becomes:
	\[
	\begin{tikzcd}[row sep = 0 em]
		0 \ar{r} & \Z/2\Z \ar{r} & \Z/2\Z  \oplus \Z/n\Z \ar{r} & H^2(K/k, K[U]^\times/ K^\times)
		\ar{r} & 0  
		\end{tikzcd}
	\]
	and we obtain the desired isomorphism.

	When $n$ is even, the sequence (\ref{LES: simplifiedS4}) becomes:
	\[
	\begin{tikzcd}[row sep = 0 em, column sep = 1 em]
		\Z/2\Z \oplus \Z/2\Z \ar{r} & \Z/2\Z  \oplus \Z/n\Z \ar{r} & H^2(K/k, K[U]^\times/ K^\times)
		\ar{r} & \Z/2\Z \ar{r} & 0
		\end{tikzcd}
	\]
	Here, the Galois group $D_n \subset S_n$ is generated by the $n$-cycle $g = (1 \,, 2 \,, \dots \,, n)$ (which under an appropriate labeling permutes the roots of $P(t)$ cyclically and fixes the element $\sqrt{a}$), and the product of transpositions $h = \prod_{2 \le i \le n-2} (i \, , \, n - 2 + i)$. The action on divisor classes is given by \smallskip
	\begin{eqnarray*}
	{g \colon \begin{cases}
	{[D_{1,i}]}  & \mapsto  {[D_{1,i+1}]}, \, 1 \le i < n-1 \\
	{[D_{1,n-1}]} & \mapsto  {-([D_{1,1}] + \dots + [D_{1, n-1}])} \\
	 \end{cases}} &
	{h \colon \begin{cases} {[D_{1,1}]} & \mapsto  {-[D_{1,1}]} \\ 
	{[D_{1,2}]} & \mapsto  {([D_{1,1}] + \dots + [D_{1,n-1}])} \\ 
	{[D_{1,i}]} &\mapsto  {-[D_{1,n-2+i}]}, \,\,\, i \le 3 \le n-2. 
	\end{cases}}
	\end{eqnarray*}

	Thus, one obtains $H^2(G,\Pic X_K) \isom \Z/n\Z$.
	\end{proof}

\section{Explicit representatives of Brauer classes} \label{section: explicit_classes}

In this section we explain the main steps involved in the construction of explicit representatives of a Brauer classes of affine Ch\^atelet surfaces. The first step is to obtain 1-cocycle representatives in $H^1(G,\Pic X_K)$. For groups $G$ which are not cyclic, we rely on so-called efficient resolutions in group cohomology.

\subsection{Efficent Resolutions} \label{subsec: efficres}

Let $G$ be a finite group and $A$ a $G$-module, with the action of $g \in G$ on $a \in A$ denoted by $g.a$. Recall that the standard resolution resolves the trivial $G$-module $\Z$ by the free $G$-modules $\Z[G^{r+1}]$ equipped with the diagonal action of $g \in G$ given by $g.(g_0, \dots, g_r) = (gg_0, \dots, gg_r)$. We denote this standard resolution by $SR$, and omit the augmentation map to $\Z$:
\[
	\begin{tikzcd}
		SR: \dots \ar{r}{\partial_3} & \Z[G^3] \ar{r}{\partial_2} \ar{r} & \Z[G^2] \ar{r}{\partial_1} & \Z[G]
	\end{tikzcd}
\]
The differential in $SR$ is defined by
$\partial_n(g_0, \dots, g_n) = \sum_{i=0}^n (-1)^i (g_0, \dots, \hat{g_i}, \dots, g_n)$, where $\hat{g_i}$ means omit the $i$-th coordinate.

In general, computations with the standard resolution can be intricate. Thus for computational purposes, alternative more efficient resolutions may be used. We denote the efficient resolution simply by $ER$. We will want to give a quasi-isomorphism of the standard complex with $ER$ by specifying morphisms of chain complexes
	\begin{eqnarray} \label{diagram: morphisms}
		\sigma = (\sigma_r)_{r = 0}^\infty: & SR \to ER \\
		\tau = (\tau_r)_{r=0}^\infty: & ER \to SR
	\end{eqnarray}
as well as a chain homotopy $h = (h_r)_{r = 1}^\infty$ satisfying $\id - \tau_r \circ \sigma_r = h_r \circ \delta_{r-1} + \delta_r \circ h_{r+1}$.
After applying the functor $\Hom_{\Z[G]}( - ,A)$, we obtain a quasi-isomorphism via maps of chain complexes $(\tau^r)_{r = 0}^\infty$, and $(\sigma^r)_{r =0}^\infty$ between $\Hom_{\Z[G]}(SR, A)$ and $\Hom_{\Z[G]}(ER, A)$ and vice versa, respectively. We will identify $\Hom_{\Z[G]}(\Z[G^{r+1}], A)$ with $\textup{Map}(G^r,A)$. 

\begin{lemma} \cite[Proposition 5]{KTdp2} \label{lem: KTdp2} 
	Let $G \isom D_n = \langle g,h \mid g^n = h^2 = ghgh = 1 \rangle$. Then there exists a projective resolution of $\Z$ 
		\begin{equation} \label{effresZGmod}
		\begin{tikzcd}
			\cdots \Z[G]^4 \ar{r}{\dd^3_n} & \Z[G]^3 \ar{r}{\dd^2_n} & \Z[G]^2 \ar{r}{\dd^1_n} & \Z[G] \ar{r}{\epsilon} & \Z
		\end{tikzcd} \end{equation}
	with
		\[
			\begin{tabular}{l l l}
				$\dd^3_n = \left( \begin{matrix} \Delta_g & 0 & 0 & N_h \\ 0 & \Delta_h & 0 & -N_g \\ 0 & 0 & \Delta_{gh}  & - N_g \end{matrix} \right),$ &
				$\dd^2_n = \left( \begin{matrix} N_g & 0 & N_{gh} \\ 0 & N_h & -N_{gh} \end{matrix} \right),$ & $\dd^1_n = \left( \begin{matrix} \Delta_g & \Delta_h \end{matrix} \right),$
			\end{tabular}
		\]
	where $N_{g_0} := 1 + g_0 + g_0^2 + \dots + g_0^{n-1}$, and $\Delta_{g_0} := 1 - g_0$ for any $g_0\in G$. Finally, $\epsilon(zg_0) = z$ for $z \in \Z$ and $g_0 \in G$. 
\end{lemma}

\begin{lemma} \label{lemma: quasi_isom} Let $G \isom D_n$, and let $SR$ and $ER$ be resolutions of $\Z[G]$ modules as above. Then there exists a quasi-isomorphism of the complexes $SR$ and $ER$.
\end{lemma}

\begin{proof} We observe that $\sigma_0 \colon \Z \to \Z$ is the identity. Thus, we have exhibited two projective resolutions of $\Z$ and an isomorphism  $\sigma_0\colon \Z \to \Z$. The comparison theorem for projective resolutions \cite[Theorem 2.2.6]{Weibel} guarantees the existence of a chain map $SR \to ER$ lifting $\sigma_0$, and that chain map is unique up to homotopy equivalence.

We define the first few morphisms $\sigma_1, \sigma_2$ for later use. It suffices to consider $\sigma_1$ on elements of $\Z[G^2]$ the form $(1, g^i h^j)$, and $\sigma_2$ on those elements of $\Z[G^3]$ of the form $(1, g^i, g^{i+i'}h^{j'})$ and $(1, g^i h , g^i h g^{i'}h^{j'})$, for $0 \le i,i' \le n-1$, and $0 \le j,j' \le 1$ due to the diagonal action of $G$. Then, 
\begin{eqnarray*} \sigma_1(1, g^ih^j) & := & \begin{cases} \big(-1 - \dots - g^{i-1},0 \big), & \textup{ if } j = 0 \\
\big(-1 - \dots - g^{i-1}, - g^i \big), & \textup{ if } j = 1
\end{cases} \\
\sigma_2(1, g^i, g^{i+i'} h^{j'}) & := & \begin{cases} (-1,0,0), & \textup{ if } i + i' \ge n \\
(0,0,0), & \textup{ otherwise } 
\end{cases} \\
 \sigma_2(1, g^i h , g^i h g^{i'}h^{j'}) & := & \Big( 0, -(j' g^{i-i'} + \sum_{e=0}^{i'-1} g^{e-i}) ,
 -\sum_{e=0}^{i'} g^{i-a+j'} h^{j'} \Big) + \begin{cases} (1,0,0), & \textup{ if } i' > i \\ (0,0,0), & \textup{ else}
\end{cases}.
\end{eqnarray*}

The higher maps may be defined similarly. See \cite{Preu} for a similar construction and further details for cyclic groups.
\end{proof} \smallskip

Given a $G$-module $M$, we may apply the contravariant functor $\Hom_G(-, M)$ to (\ref{effresZGmod}) obtain a cochain complex, where the dual maps denoted as above with the subscripts and superscripts interchanged. That is, we identify $\Hom_G(\Z[G], M)$ with the module $M$ itself, which then yields the cochain complex
	\[
		\begin{tikzcd}
			0 \ar{r} & M \ar{r}{\dd_1^n} & M^2 \ar{r}{\dd_2^n} & M^3 \ar{r}{\dd_3^n} & M^4 \ar{r} & \dots
		\end{tikzcd}
	\]

and we may identify classes in $H^1(G, M) \cong \ker(\dd_2^n)/ \im(\dd_1^n)$.

\begin{example} \label{example: XmBrauer}
We apply the theory to the surface $\calX_{m}$ for $m$ not a fourth power in $\Q$, and $M = \Pic X_K$.
Here, $G = \Gal(K/k) \isom D_4$, and $\sqrt{-a} = i \in K$. In this case, we recall that $\Pic \Xbar \isom \Pic X_K$ is a free $\Z$-module of rank 3, generated by the Galois invariant set of divisors $\{D_{i,j}\}$. We shall make the convention that $e_1 = \sqrt[4]{m}$. We take $\{[D_{1,1}], [D_{1,2}], [D_{1,3}]\}$ as a basis for $\Pic X_K$.

The action of the Galois group on $\Pic X_K$ is given by 
	\[ g = \begin{pmatrix} \,\,\,\,0 & \,\,\,\,0 & \,\,\,\,1 \,\, \\ -1 & -1 & -1 \,\, \\ \,\,\,\,0 & \,\,\,\,1 & \,\,\,\,0 \,\,  \end{pmatrix} \hspace{.5 em}\textup{ and } \hspace{.5 em} h = \begin{pmatrix} -1 & \,\,\,\,0 & \,\,\,\,0 \,\,\\ \,\,\,\,0 & -1 & \,\,\,\,0 \,\,\\ \,\,\,\,1 & \,\,\,\,1 & \,\,\,\,1\,\, \end{pmatrix} \]

To compute $H^1(G, \Pic X_K)$ for $\calX_{m}$ we make use of the efficient resolution above. Letting $v = (v_1, v_2, v_3)$, and $v'= (v_1', v_2', v_3') \in \Pic X_K$, the condition that $(v,v')$ represent a class in $H^1(G, \Pic X_K)$ translates to
	\begin{equation}\label{relations}
	N_g v = N_h v' =0 \text{ and } N_{gh} v = N_{gh} v'
	\end{equation}
	modulo those of the form $(\Delta_g w, \Delta_h w)$. One obtains the relations:
	\begin{multicols}{4}
	\begin{enumerate}
		\item $v_3' = 0 $
		\item $v_2 = v_2' $
		\item $v_1 - v_3 = v_1'$
	\end{enumerate}
	\end{multicols}
	Hence a 1-cocycle representative is given by a pair $(v,v')$ such that $v = (v_1, v_2, v_3)$ and $v' = (v_1 - v_3, v_2, 0)$, modulo those pairs of the form $(\Delta_g w, \Delta_h w)$, where
	\begin{enumerate}
		\item $\Delta_g w = (w_1 + w_2, 2w_2 - w_3, -w_1 + w_2 + w_3) $
		\item $\Delta_h w = (2w_1 - w_3, 2w_2 - w_3, 0)$
	\end{enumerate}
	One finds that $H^1(G,M) \isom \Z/4\Z$ and is generated by the pairs $(v,v')$ and $(\tilde{v}, \tilde{v}')$ where
	\begin{eqnarray*}
		v = & (1,1,1) & = [D_{1,1}] + [D_{1,2}] + [D_{1,3}]\\
		v' =& (0,1,0) & = [D_{1,2}] \\
		\tilde{v} = &(1,0,0) & = [D_{1,1}] = \tilde{v}'.
	\end{eqnarray*}
\end{example} \smallskip

\subsection{Cyclic Algebras}

For a smooth variety $X$ over a global field $k$ and a Galois extension $L/k$, we write $N_{L/k} \colon \Div X_L \to \Div X_k$, and $N_{L/k} \colon \Pic X_L \to \Pic X_k$, for the usual norm maps on $\Div X$ and $\Pic X$, respectively. The following proposition characterizes the cyclic algebras in the image of the map $\Br X \to \Br \kk(X)$. 
\begin{prop}[\cite{Corn}] Let $X$ be a smooth geometrically integral variety over a field $k$. Let $L/k$ be a finite cyclic extension and let $f \in \kk(X)$. Then the
cyclic algebra $(L/k,f)$ is in the image of the natural map $\Br X \hookrightarrow \Br \kk(X)$ if and only if $(f) = N_{L/k}(D)$ for some $D \in \Div(X_L)$. Furthermore, if $X$ is locally soluble, then $(L/k,f)$ comes from $\Br k$ if and only if $D$ can be taken to be principal.
\end{prop}

In particular, in the case of surfaces in Theorem \ref{thm: short version general dihedral}, we observe that since dihedral extensions of order $2n$ contain a unique cyclic subextension of order $n$ (which is never a degree $n$ extension of the base field), the generator of $\Br X/ \Br k$ cannot easily be seen to be cyclic, and in fact may not be representable in this way. One can apply this theory to obtain a representative of a cyclic algebra generating the Brauer group after a base change instead. However, when $n$ is even, one cannot simply corestrict to obtain a generator over the ground field. Moreover, in the case of $\calX_{22}$ for example, base change to the relevant quadratic extension $\Q(i)$ already witnesses the existence of $\Z[i]$ points (e.g. $(1,5,(1+i))\,$), and thus the presence of a Brauer-Manin obstruction does not commute with quadratic base change. Thus, we search for an alternative approach to constructing explicit representatives of Brauer classes for these surfaces.

\subsection{Brauer classes via 2-cocycles in Galois cohomology} \label{subsec: H2BrU} 
Based on ideas from \cite{KT08}, given a 1-cocycle representative for a class in 
$H^1(G,\Pic X_K)$, we shall describe an algorithm for lifting to a 2-cocycle representative in $H^2(G, \calO(U_K)^\times)$ that is the restriction of a generator of $\Br X$. We first need the data of the following sequences of $G$-modules.

	\begin{lemma} \label{lem: 3exactseqs} Let $X$ be an affine Ch\^atelet surface, as in Theorem \ref{thm: extended brauer deg 4}. Then the following sequences of $G$-modules are exact.
		\begin{enumerate}
			\item Under the identification $\Pic \Xbar \cong \Pic X_K$, 
				\begin{equation} \label{PicExact}
				\begin{tikzcd} 0 \ar{r} & R \ar{r} & \displaystyle \bigoplus_{i,j=1}^n  \Z [D_{i,j}]  \ar{r} & \Pic(X_K) \ar{r} & 0
				\end{tikzcd}
				\end{equation}
				where $R$ is the subgroup of relations.

			\item From the Hochshild-Serre spectral sequence $H^p(G, H^q_{\et}(X_K, \G_m)) \implies H^{p+q}_{\et}(X, \G_m)$, we obtain the low degree terms of the long exact sequence 
				\begin{equation} \label{HSSpectral}
					\begin{tikzcd}[column sep = 1 em] \Br(K/k) \ar[hookrightarrow]{r} & \ker(\Br X \to \Br X_K) \ar{r}  & H^1(G, \Pic X_K) \ar{r} & H^3(G,K^\times)
					\end{tikzcd}
				\end{equation}
				where $\Br(K/k):= \ker(\Br k \to \Br K)$.
			\item  Let $U = X \setminus \{D_{i,j}\}$, and $R$ as above. Then, 
				\begin{equation} \label{UExact} 
					\begin{tikzcd} 0 \ar{r} & K^\times \ar{r} & \calO(U_K)^\times \ar{r}{\divv} \ar{r} & R \ar{r} & 0
					\end{tikzcd}
				\end{equation}
		\end{enumerate}
	\end{lemma}

 Thus, given a 1-cocycle $\Btilde \in H^1(G, \Pic X_K)$, we apply the connecting homomorphism of the long exact sequence in cohomology of (\ref{PicExact})  to the cocycle representative to obtain a 2-cocycle representative of the corresponding element in $H^2(G,R)$. Then, extending $K$ if necessary, we kill the obstruction in $H^3(G, K^\times)$ to lifting the element of $H^2(G, R)$ to an element $B \in H^2(G, \calO(U_K)^\times)$. We carry out the lifting explicitly on the level of cocycles. As we shall soon demonstrate, the element $B$ will be the restriction of an element $\calA \in \Br X$ which agrees with the image of the class $\Btilde$ in $\Br U$ up to a constant algebra in $\Br k$. \smallskip

	We will also be interested in when the sequence (\ref{UExact}) has a splitting, and will return to this in the next section in a particular application.
	\begin{claim} \cite[Proposition 2]{DSW15} \label{claim: splitting} The sequence (\ref{UExact}) is split for (affine) norm form varieties if and only if the classes can be lifted to $\calO(U_K)^\times$ in a Galois-equivariant way, via a map 
	\[ \phi: \calO(U_K)^\times/ K^\times \to \calO(U_K)^\times, \] 
	defined by $[t - e_i] \mapsto \rho^{-1}(t - e_i)$ and $[u_1] \mapsto \xi^{-1} u_1$, where $\rho \in k(e_1)^\times$, and $\xi \in k(\sqrt{a})^\times$. Because of the defining equation of $X$, which determines the unique relation on the invertible functions on $U_K$, the pair $(\rho, \xi)$ defines a splitting if and only if 
	\[ c N_{k(e_1)/k}(\rho) = N_{k(\sqrt{a})/k}(\xi). \] 
	\end{claim} \smallskip
 
	Combining the cohomology exact sequences arising from the sequences (\ref{PicExact}), (\ref{UExact}), and a portion of the sequence (\ref{HSSpectral}), one obtains a diagram:
	\begin{equation} \label{CohomDiagram} 
		\begin{tikzcd}[column sep = 1.2 em]
			{ } & {} & { } & \Br(U) \\
			{} &  \ker(\Br (X) \to \Br (X_K)) \ar{urr} \ar{d} & H^2(G, \calO(U_K)^\times) \ar{ur} \ar{d}{\mu} & {} \\
			0 \ar{r} & H^1(G, \Pic X_K) \ar{r}{\delta} \ar[swap]{dr}{\epsilon} & H^2(G,R) \ar{d}{\nu} \ar{r} & H^2(G, \displaystyle \bigoplus_{i,j} \Z [D_{i,j}]) \\
			{ } &  { } &  H^3(G, K^\times) & { }
		\end{tikzcd}
	\end{equation} \smallskip

	\begin{remark}
		In this diagram, as in \cite{KT08}, we have used the fact that $\Div_{X_K \setminus U_K}(X_K) = \bigoplus_{i,j} \Z [D_{i,j}]$ is a permutation module, hence its first cohomology vanishes. The maps to $\Br(U)$ are respectively the restriction map from $\Br X$ and the map from $H^2(G,\calO(U_K)^\times)$ which sends the class of a 2-cocycle with values in $\calO(U_K)^\times$ to the class in $\Br U$ represented by the same cocycle, viewed as a \^Cech cocycle for the covering $U_K \to U$.
	\end{remark} \smallskip

	\begin{proposition} Let $X$ be an affine Ch\^atelet surface. The map $\epsilon$ in diagram (\ref{CohomDiagram}) is the rightmost map in the sequence (\ref{HSSpectral}).
	\end{proposition}

	\begin{proof} We proceed as in \cite{KT08}, with some additional details added in for exposition.
		Let $j: U \to X$ be the inclusion, and let $D$ be the complement of $U$ in $X$. There is an exact sequence of \`etale sheaves on $X$ (\cite{CTS87})
		\begin{equation} \label{resGm}
			\begin{tikzcd} 0 \ar{r} & \G_m \ar{r} & j_\ast(\G_m|_U) \ar{r}& \calZ^1_D \ar{r} & 0,
			\end{tikzcd}
		\end{equation}
		where $\calZ_D^1$ is the sheaf of divisors on $X$ with support on $D$. By evaluating global sections on $X_K$ (i.e. by applying the functor $\Hom_{X_{\et}}(\hat{\G}_m, - ))$, and forming the long exact sequence in cohomology, we obtain
	\begin{equation*} \small
	\begin{tikzcd} [column sep = small]
	H^0(X_K, \G_m) \ar[hookrightarrow]{r} & H^0(U_K,\G_m) \ar{r} & \Hom_X(\hat{\G}_m, Div_{D}X)
 	\ar{r} & H^1(X_K, \G_m) \ar{r} & H^1(U_K, \G_m) \ar{r} & \cdots .
 	\end{tikzcd} 
 	\end{equation*}
 	Under the identifications,
 	\begin{enumerate}[(i)]
 		\item $H^0(X, \G_m) \isom \calO_X(X)^\times = K^\times$
 		\item $H^0(U_K, \G_m) \isom  \calO(U_K)^\times $
 		\item $H^1(X_K, \G_m) = \Pic X_K $
 		\item $H^1(U_K, \G_m) = \Pic U_K $ 
 		\item $\Hom(\hat{\G}_m, \Div_{D}X) = \Div_{X_K \setminus U_K}(X_K),$
 	\end{enumerate}
 	we obtain the exact sequence of $G$-modules
 	\begin{equation} \label{4termExact}
	\begin{tikzcd}
		0 \ar{r}&  K^\times \ar{r} & \calO(U_K)^\times \ar{r} & \Div_{X_K \setminus U_K}(X_K) \ar{r} & \Pic(X_K) \ar{r} & 0 \,. 
	\end{tikzcd} 
	\end{equation}

 	The sequence (\ref{resGm}) provides a resolution of the sheaf $\G_{m}$ on $X$. Let $(\calI_{\bullet}, \dd)$ be a resolution of $\G_m$ by injective sheaves on $X$. Then by the comparison theorem \cite[Theorem 2.2.6]{Weibel}, there exists a morphism of resolutions from the resolution (\ref{resGm}) to $\calI_\bullet$. 
	The \'etale cohomology groups $H^i_{\et}(X_K, \G_m)$ are computed as
	\begin{equation}  \label{injectiveres}
		H^i[ \begin{tikzcd} 
		0 \ar{r} & H^0(X_K, \calI_0) \ar{r}{d^0} & H^0(X_K, \calI_1) \ar{r}{d^1} &  H^0(X_K, \calI_2) \ar{r}{d^2} & \dots 
		\end{tikzcd}]
	\end{equation}

	The Picard group is given by $\Pic(X_K) \isom H^1_{\et}(X_K, \G_m) = \ker{ \dd^1}/ \im \dd^0.$ Additionally,
	$\im(H^0(X_K,\calI_0) \to H^0(X_K, \calI_1)) = H^0(X_K, \calI_0)/ H^0(X_K, \G_m)$, and we've previously observed that $
	H^0(X_K, \G_m) = K^\times$. Thus, we obtain the 4-term exact sequence
	\[ \begin{tikzcd} 
		K^\times \ar{r} & H^0(X_K, \calI_0) \ar{r} & \ker \left(H^0(X_K, \calI_1) \to H^0(X_K, \calI_2)\right) \ar{r} & \Pic(X_K) 
		\end{tikzcd}
	\]
	and hence the following diagram

	\begin{equation} \label{4termdiagram}
		\begin{tikzcd} 
			K^\times \ar[hookrightarrow]{r} \ar[equal]{d} & \calO(U_K)^\times \ar{r}{\divv} \ar{d}{\psi} & \Div_{X_K \setminus U_K} X_K \ar[->>]{r} \ar{d}{\varphi}& \Pic(X_K) \ar[equal]{d} \\
			K^\times \ar[hookrightarrow]{r} & H^0(X_K, \calI_0) \ar{r}{\dd^0} & \ker(\dd^1) \ar[->>]{r} & \Pic(X_K),
		\end{tikzcd}
	\end{equation}
	where the first and last maps are identity maps.

	The four term exact sequence (\ref{4termExact}) is the amalgamation of two short exact sequences. The map $\epsilon$ in diagram (\ref{CohomDiagram}) is the composition of the connecting homomorphisms of the two long exact sequences in cohomology. That is, 
	\[ 
		\begin{tikzcd} 
		H^1(G, \Pic X_K) \ar{r}{\delta} \ar[bend left = 20]{rr}{\epsilon} & H^2(G,R) \ar{r}{\nu} & H^3(G, K^\times) 
		\end{tikzcd}
	\]

	We claim that the edge map in (\ref{HSSpectral}) is equal to a similar composition of connecting homomorphisms coming from the 4-term exact sequence of the bottom row of diagram (\ref{4termdiagram}). In particular, we claim that the edge map is the composition
	\[ H^1(G, \Pic X_K) \to H^2(G, \im(\dd^0)) \to H^3(G,K^\times) \] 

	This follows from the construction of the differentials on the $E_2$-page of the Hochschild-Serre spectral sequence in \'etale cohomology, for the Galois covering $X_K \to X$. There is a connecting homomorphism that is constructed as in the snake lemma, and this produces the desired map $d_2^{1,1} \colon E^{1,1}_2 = H^1(G, \Pic X_K) \to E^{3,0}_2 = H^3(G, K^\times)$.

	Then, as there exists a map between these sequences inducing the identity maps on the first and last terms, the morphism $\epsilon$ is equal to the edge map of (\ref{HSSpectral}).
	\end{proof}

	\begin{proposition} \label{prop: KT08 II} Let $X$ satisfy the conditions of Lemma \ref{lem: 3exactseqs} and let $\Btilde \in H^1(G,\Pic X_K)$ with $\epsilon(\Btilde) = 0$. Let $\calA \in \ker(\Br X \to \Br X_K)$ be a lift of $\Btilde$ by the map $\lambda$ and let $B \in H^2(G, \calO(U_K)^\times)$ be a lift of $\delta(\Btilde)$ by the map $\mu$ in (\ref{CohomDiagram}). Then the images of $\calA$ and $B$ in $\Br U$ via the maps in (\ref{CohomDiagram}) are equal modulo a constant algebra in $\Br k$.
	\end{proposition}

	\begin{proof} Let $\Btilde \in H^1(G, \Pic X_K)$ be a generator, with 1-cocycle representative $\widetilde{\beta}$. Suppose that $\epsilon(\Btilde) = 0$. In order to compute $\delta(\Btilde)$, we must first lift $\betatilde$ to a 1 cochain $\gammatilde$ with values in $\bigoplus_{i,j} \Z[D_{i,j}]$. An application of the snake lemma for the complex (\ref{PicExact}) provides us with the desired class $\delta(\Btilde) \in H^2(G,R)$. Now, we consider the following diagram of cochains with exact rows, induced by the the 4-term exact sequence (\ref{4termExact}), where we let $M := \Div_{X_K \setminus U_K}(X_K)$:

 	\begin{equation}\label{CochainsDiag} \small
 		\begin{tikzcd}[column sep = small]
 		C^i(G, K^\times)\ar[hook]{r} \ar{d}{\partial}& C^i(G,\calO(U_K)^\times) \ar{d}{\partial} \ar{r}{\divv} & C^i(G,M) \ar{r}{\pi} \ar{d}{\partial} & C^i(G, \Pic(X_K)) \ar{d}{\partial} \ar[dotted, bend right = 2]{dll}\\
 		C^{i+1}(G, K^\times) \ar[hook]{r} & C^{i+1}(G,\calO(U_K)^\times) \ar{r}{\divv} & C^{i+1}(G,M) \ar{r}{\pi} & C^{i+1}(G, \Pic(X_K))
 		\end{tikzcd}
 	\end{equation}

 	By assumption $\pi(\widetilde{\gamma}) = \widetilde{\beta}$, and $\widetilde{\beta} \in Z^1(G, \Pic X_K)$, hence $\partial(\widetilde{\beta}) = 0$. Thus $\pi \circ \partial(\widetilde{\gamma}) = 0$. That is, $\partial(\widetilde{\gamma}) \in \ker \pi$, thus by exactness of the bottom row of (\ref{CochainsDiag}), $\partial(\widetilde{\gamma}) = \divv(\beta)$ for some 2-cochain $\beta \in C^2(G, \calO(U_K)^\times)$. Moreover, we observe that $\partial(\beta)$ can be identified with a 3-cocycle representative of $\epsilon(\Btilde)$, which vanishes by assumption. By modifying $\partial(\beta)$ by a 2-cochain in $C^2(G,K^\times)$, we may assume that $\partial(\beta) = 0$. Hence $\beta$ is a 2-cocycle representative of a class $B \in H^2(G,\calO(U_K)^\times)$ such that $\delta(\Btilde) = \mu(B)$.

	We may identify $\Br X$ with the cohomology of the total complex of the spectral sequence $C^p(G, H^0(X_K, \calI_q)$. Then we claim that a class $\calA \in \Br X$ is represented by 2-cocycle of the total complex 
	\[(\alpha_0, \alpha_1, \alpha_2) \in C^0(G, H^0(X_K, \calI_2)) \oplus C^1(G, H^0(X_K, \calI_1)) \oplus C^2(G, H^0(X_K, \calI_0). \]
	Furthermore, if $\calA \in \ker(\Br X \to \Br X_K)$, is a lift of $B$, then $\calA$ has a cocycle representative of the form $(0, \varphi(\gammatilde), \alpha_2)$.

	To see this, we form the double complex, by first considering the resolution $(\calI_\bullet, \dd)$ which forms a complex of $G$-modules, and then we resolve each term horizontally: 
		\begin{equation} \label{bicomplex} \small
			\begin{tikzcd}
			\vdots & \vdots & \vdots & { } \\ 
			C^0(G,H^0(X_K, \calI_2)) \ar{r}{\partial^{0,2}} \ar{u}{\dd^{0,2}} & C^1(G, H^0(X_K, \calI_2)) \ar{r}{\partial^{1,2}} \ar{u}{\dd^{1,2}} & C^2(G, H^0(X_K, \calI_2)) \ar{r}{\partial^{2,2}}\ar{u}{\dd^{2,2}} & \dots  \\
			C^0(G,H^0(X_K, \calI_1)) \ar{r}{\partial^{0,1}} \ar{u}{\dd^{0,1}} & C^1(G, H^0(X_K, \calI_1)) \ar{r}{\partial^{1,1}} \ar{u}{\dd^{1,1}} & C^2(G, H^0(X_K, \calI_1)) \ar{r}{\partial^{2,1}}\ar{u}{\dd^{2,1}} & \dots \\
			C^0(G,H^0(X_K, \calI_0)) \ar{r}{\partial^{0,0}} \ar{u}{\dd^{0,0}} & C^1(G, H^0(X_K, \calI_0)) \ar{r}{\partial^{1,0}} \ar{u}{\dd^{1,0}} & C^2(G, H^0(X_K, \calI_0)) \ar{r}{\partial^{2,0}}\ar{u}{\dd^{2,0}} & \dots 
			\end{tikzcd}
		\end{equation} \smallskip

	To compute the cohomology of the total complex, we define modules
	\[ \Tot^{m} := \displaystyle \bigoplus_{i+ j = m} \Tot^{ij} = \displaystyle \bigoplus_{i+j = m} C^i(H^0(X_K, \calI_j))\]

	Let $\D^m \colon \Tot^m \to \Tot^{m+1}$ be defined by 
	\[\D^m(c^{ij}) = \dd^{j}(c^{ij}) + (-1)^j \partial^{i,j}(c^{ij})\]

	It is a standard computation that $D^{m+1} \circ D^m = 0$, so that $(\Tot^\bullet, \D^\bullet)$ forms a complex of $G$-modules. Now, $\calA \in \Br X$ is represented by a cocycle of the total complex 
	\[ (\alpha_0, \alpha_1, \alpha_2) \in C^0(G, H^0(X_K, \calI_2)) \oplus C^1(G, H^0(X_K, \calI_1)) \oplus C^2(G, H^0(X_K, \calI_0)) \]
	such that, $\dd^{0,2}(\alpha_0) = 0$, $\partial^{0,2}(\alpha_0) = -\dd^{1,1}(\alpha_1)$, $\partial^{1,1}(\alpha_1) = \dd^{2,0}(\alpha_2)$, and $\partial^{2,0}(\alpha_2) = 0$.

	Since $\calA \in \ker(\Br X \to \Br X_K)$, and $\Br X_K = \ker \dd^{0,2}/ \im \dd^{0,1}$, $\calA$ is represented by an $(\alpha_0, \alpha_1, \alpha_2)$ such that $\alpha_0 \in \im \dd^{0,1}$, say $\alpha_0 = \dd^{0,1}(\beta_0)$. We claim $\calA$ has a cocycle representative of the form $(0, \alpha_1', \alpha_2)$. That is, $(\alpha_0, \alpha_1 - \alpha_1', 0)$ is a coboundary for the total complex so is in the image of $D^1$. Indeed, $D^1(( \beta_0, 0)) = (\alpha_0, \partial^{0,1}(\beta_0), 0)$. Hence $\calA$ is represented by the 2-cocycle $(0, \alpha_1 - \partial^{0,1}(\beta_0), \alpha_2)$, which we shall rename $(0, \alpha_1, \alpha_2)$.

	Since (\ref{4termdiagram}) is a diagram of $G$-modules, we may form the complex
	\begin{equation}
		\begin{tikzcd}
			C^1(G,\calO(U_K)^\times) \ar{r}{\divv} \ar{d}{\psi} & C^1(G,\Div_{X_K \setminus U_K} X_K) \ar[->>]{r} \ar{d}{\varphi}& C^1(G,\Pic(X_K)) \ar[equal]{d} \\
			C^1(G,H^0(X_K, \calI_0)) \ar{r}{\dd^0} & C^1(G, \ker(\dd^1)) \ar[->>]{r} & C^1(G,\Pic(X_K)),
		\end{tikzcd}
	\end{equation}

	Since $\calA$ is a lift of $\Btilde$, we have that $\alpha_1$ is a lift of $\betatilde$ in the bottom row. By assumption, $\gammatilde$ is a lift of $\betatilde$ in the top row. Thus, $\alpha_1 - \varphi(\gammatilde)$ maps to $0$ in $C^1(G, \Pic X_K)$, hence is in the image of $d^{1,0}$, say $\alpha_1 - \varphi(\gammatilde)= d^{1,0}(\beta_1)$. It suffices to show that $(0, \alpha_1 - \varphi(\gammatilde), \alpha_2')$ is in the image of $D^1$. Indeed, $D^1( (\beta_1, 0)) = (0,\alpha_1 - \varphi(\gammatilde), \partial^{1,0}(\beta_1))$. Thus, $\calA$ has a cocycle representative of the form $(0, \varphi(\gammatilde), \alpha_2 - \partial^{1,0}(\beta_1))$, as desired.

	The condition to be a 2-cocycle simplifies in the case when $\alpha_0 = 0$ to the following: 
	\begin{eqnarray*}
		\dd^{1,1}(\alpha_1) &=& 0 \\
		\partial^{1,1}(\alpha_1) & = & \dd^{2,0}(\alpha_2) \\
		\partial^{2,0}(\alpha_2) & = & 0 
	\end{eqnarray*}

	Again, diagram (\ref{4termdiagram}) gives us the commutative square 
	\[ \begin{tikzcd} C^2(G, \calO(U_K)^\times) \ar{r}{\divv} \ar{d}{\psi} & C^2(G, \Div_{X_K \setminus U_K}(X_K)) \ar{d}{\varphi} \\
	C^2(G, H^0(X_K, \calI_0)) \ar{r}{\dd^{2,0}} & C^2(G, \ker \dd^{2,1})
	\end{tikzcd}
	\]

	We observe that since $\calA$ has a cocycle representative $(0, \varphi(\gammatilde), \alpha_2)$ that satisfies the above conditions,
	\begin{eqnarray*}
	\dd^{2,0}(\psi(\beta)) &=& \varphi(\divv(\beta)) \\
	& = & \varphi( \partial(\gammatilde ))\\
	& = & \partial^{1,1}(\varphi(\gammatilde)) \\
	& = & d^{2,0}(\alpha_2)
	\end{eqnarray*}

	Hence $\psi(\beta) - \alpha_2 \in \ker(\dd^{2,0}) = H^2(G, K^\times)$. Moreover, on $U$, $\varphi(\gammatilde)$ vanishes, thus the image of $\calA - B \in \Br U$ is represented in the total complex by $(0, 0, (\alpha_2 - \psi(\beta))|_U)$, hence by the above analysis is given by a 2-cocycle with coefficients in $\ker \dd = K^\times$i.e., an element of $\Br k$.

\end{proof}

As mentioned earlier, it is possible that one might need to extend the field $K$ to kill the obstruction to lifting given by $H^3(G,K^\times)$. The following lemma and its corollary characterize certain number fields for which this group is automatically trivial.

\begin{lemma} {\cite[Theorem 2.119]{Koch}} \label{prop: Koch} Let $L/k$ be an extension of number fields with Galois group $G$. Then $H^3(G, L^\times)$ is cyclic of order $[L:k]/ \lcm\{ [L_{w_v} : k_v] \mid v \in \Omega_k \}$. 
	\end{lemma}

	\begin{cor} \label{cor: Koch} If $L/k$ is an extension of number fields, and $v \in \Omega_k$ is a place of $k$ such that $[L:k] = [L_{w_v}:k_v]$, then $H^3(G, L^\times) = 0$. 
	\end{cor}

We are now able to apply the tools from this section to give explicit 2-cocycle representatives of Brauer classes generating the Brauer groups of affine Ch\^atelet surfaces. 

\begin{theorem} \label{thm: BrX_explicitrep}
	Let $X$ be an affine Ch\^atelet surface as in Theorem \ref{thm: short version general dihedral} with $\Br X/\Br_0 X \isom \Z/n\Z$. Let $e_1$ be a root of $P(t)$ fixed by the automorphism $\sqrt{a} \mapsto -\sqrt{a}$.
	\begin{enumerate}
		\item If $P(t)$ has leading coefficient $c$ such that $N_{L/k}(c') = c$ then one may reduce to the case when $P(t)$ is monic, and in that case $\Br X/\Br_0 X$ is generated by
		\label{case: monic}
		\[ \Big(x + \sqrt{a}y, t- e_1,1\Big) \in H^2(G, \calO(U_K)^\times)\]
		\item If $P(t)$ has leading coefficient $-1$ then,
			\begin{enumerate}[(i)]
				\item If $\deg P(t)$ is odd, then the same representative as in (\ref{case: monic}) generates $\Br X/ \Br_0 X$.
				\item $X_m\colon x^2 + y^2 = -(t^4 - m)$, $m \in \Z_{\ge 0}$, has $\Br X_m/\Br_0 X_m$ generated by 
				\[ \Big(-2m(x + iy), e_1(t-e_1), (1+i)e_1\Big) \in H^2(G, \calO(U_K)^\times)\]
			\end{enumerate}
	\end{enumerate}
\end{theorem}

\begin{proof} We prove (2)(ii); the remaining cases follow similarly. As in Example \ref{example: XmBrauer}, one computes using the efficient resolution (\ref{effresZGmod}) that in each of the above cases, a generator for $H^1(G, \Pic X_K)$ is given by the pair $([D_{1,1}], [D_{1,1}]) \in (\Pic X_K)^{2}$. The next step is to lift this to a 1-cochain representative in $H^1(G, \Div_{X_K\setminus U_K}(X_K))$; one may take $(D_1, D_1) \in (\Div_{X_K\setminus U_K}(X_K))^2$. Next, map via the differential from the efficient resolution to obtain a 2-cocycle in $H^2(G, \Div_{X_K\setminus U_K}(X_K))$, which is represented by a triple in $(\Div_{X_K\setminus U_K}(X_K))^3$, given by
{\small 
\begin{eqnarray*} \begin{pmatrix} D_1, & D_1 \end{pmatrix} \begin{pmatrix} N_g & 0 & N_{gh} \\ 0 & N_h & -N_{gh} \end{pmatrix} & = &  \begin{pmatrix} N_g(D_1), & N_h(D_1), & N_{gh}(D_1) - N_{gh}(D_1) \end{pmatrix} \\
	& = & \begin{pmatrix} D_{1,1} + D_{1,2} + D_{1,3} + D_{1,4}, & D_{1,1} + D_{2,1}, & 0 \end{pmatrix}
	\end{eqnarray*}}
This determines a class in $H^2(G, R)$, represented by the same triple. The methods of this section suggest that to find the 2-cocycle representative that lifts this class to a class in $H^2(G, \calO(U_K)^\times)$, one must find functions on $U$ that satisfy the same cocycle relations arising from the efficient resolution. That is, we require functions $(f_1, f_2, f_3) \in \calO(U_K)^\times$ such that $\divv(f_1) = N_g([D_{1,1}])$, $\divv(f_2) = N_h([D_{1,1}])$ and $\divv(f_3) = N_{gh}([D_{1,1}] - [D_{1,1}])$. Thus in each case, one computes that the listed functions satisfy the desired properties. Additionally, the condition that $(f_1, f_2, f_3)$ defines a 2-cocycle in $H^2(G, \calO(U_K)^\times)$ means that $N_h(f_1) = N_g(f_2 f_3)$, and one also checks that these holds for the functions specified in the statement of the theorem.
\end{proof}


\section{Effective cocycle lifting} \label{sec: cocycle_lifting}

In this section, we provide a procedure to evaluate local invariants of the non-cyclic algebras of \S\ref{section: explicit_classes}. First, we discuss more broadly an approach for computing the local invariant of a 2-cocycle representing a Brauer class over a local field.
Let $k_v$ be the completion of $k$ at $v$. Since $K/k$ is Galois, there exists only one place $w$ of $K$ extending $v$, up to Galois conjugacy. Denote the completion by $K_w$. Define $Q_1 := \Gal(K_w/k_v)$, and let $\varphi \in H^2(Q_1, K_w^\times)$ be a 2-cocycle.

For an infinite place, the Brauer group is either trivial or isomorphic to $\Z/2\Z$, hence to compute the invariant it suffices to determine whether the class vanishes.

If $v$ is finite, let $d$ dividing $|Q_1|$ satisfy $d [\varphi] =0$, e.g. $d = |Q_1|$ \cite{Serre}. Let $\kcycl$ be the unique unramified extension of degree $d$ up to isomorphism, effectively constructible as a cyclotomic extension, and let $v_d$ denote the corresponding valuation. Let $A = \kcycl \cdot K_w$ be the compositum of fields with Galois group $\Gtilde := \Gal(A/k_v)$. We obtain the following diagram of fields:
\begin{equation} \label{eqn: diagram_of_fields}
	\begin{tikzcd}
		{} & A \ar[dash,swap]{dl}{H_2} \ar[dash]{d}{H_1}  \ar[dash,swap]{ddl}{\Gtilde}\\
		\kcycl \ar[dash,swap]{d}{Q_2 \isom \Gtilde/H_2} & K_w \ar[dash]{dl}{Q_1 \isom \Gtilde/H_1}\\
		k_v & {}
	\end{tikzcd}
\end{equation}

This corresponds to the diagram in Galois cohomology:

\begin{equation} \label{eqn: inf_res_diag}
	\begin{tikzcd}
		{} & H^2(Q_2, (A^{H_2})^\times) \ar{d}{\inf_2} & {} \\
		H^2(Q_1, (A^{H_1})^\times) \ar{r}{\inf_1} & H^2(\Gtilde,A^\times) \ar{d}{\res_2} \ar{r}{\res_1} & H^2(H_1, (A |_{H_1})^\times)^{Q_1} \\
		{} & H^2(H_2, (A |_{H_2})^\times)^{Q_2} &{}
	\end{tikzcd} 
\end{equation}

By Hilbert's Theorem 90, the corresponding first Galois cohomology groups in the diagram vanish, thus we are left with the higher 5-term inflation-restriction sequences, of which the first three terms appear in the row and column of this diagram. Since $d [\varphi] = 0$, so too is $\res_2 \inf_1 [\varphi] = 0$. That is, 
\begin{equation} \label{eqn:infres}
	0 = d \inv_v (\textup{inf}_1 [\varphi]) = [\kcycl: k_v] \inv_v (\textup{inf}_1 [\varphi]) = \inv_v (\textup{res}_2 (\textup{inf}_1[\varphi])). 
\end{equation}

\begin{lemma} 
	Let $[\varphitilde]$ be the lift of $\inf_1 [\varphi]$ to $H^2(Q_2, \kcycl^\times)$. Then
	\[ 
		\inv_v[\varphi] =  \inv_v[\varphitilde] = \frac{1}{d} v_d \left( \varphitilde( \displaystyle \prod_{j=0}^{d-1} \rho^j, \rho) \right),
	\] 
	where $\rho \in Q_2$ induces Frobenius on the residue field, $\overline{\kcycl}$.
\end{lemma}

\begin{proof}
	We follow the construction of the crossed-product algebra in \cite{MilneCFT}. Let $\calA(\varphitilde)$ be the $\kcycl$ vector space with basis $(e_{\rho^i})_{0 \le i \le d-1}$ endowed with multiplication given by 
	\begin{enumerate}
		\item $\rho^i a = e_{\rho^i}ae_{\rho^i}^{-1}$ for all $a \in \kcycl^\times$
		\item \label{cprodalgmult} $e_{\rho^i}e_{\rho^j} = \varphitilde(\rho^i, \rho^j)e_{\rho^{i+j}}$ for all $0 \le i,j \le d-1$.
	\end{enumerate}
	Then $\inv_v[\varphitilde] = \inv_v(A (\varphitilde)) := \ord(e_\rho) \mod \Z$, where $\ord(\alpha)$ for $\alpha \in A(\varphitilde)$ is defined by 
	\[ 
		|\alpha|_v = \left(\frac{1}{\# \overline{\kcycl}}\right)^{\ord( \alpha)}.
	\]
	Since $\kcycl$ is an unramified extension of $k_v$,
	$\ord(e_\rho) \mod \Z = \frac{1}{d}\ord(e_\rho^d) \mod \Z = \frac{1}{d} v_d(e_{\rho}^d) \mod \Z$, and by (\ref{cprodalgmult}), we find
	\[
		\inv_v[\varphitilde] = \frac{1}{d} v_d(e_{\rho}^d) = \left( \varphitilde( \prod_{j=0}^{d-1} \rho^j, \rho) \right) \in {1/d\Z/\Z}
	\]
	as desired.
\end{proof}

The explicit computation of the local invariant thus relies on effective lifting of two-cocycles. Since $\inf_1 [\varphi] = \inf_2 [\varphitilde] \in H^2(\Gtilde,A)$, there exists a 1-cochain $\psi \in C^1(\Gtilde,A)$ satisfying the following conditions:
\begin{enumerate} \label{cocycle_conditions}
	\item $\inf_1(\varphi) = \partial \psi \cdot \inf_2(\varphitilde)$
	\item $\res_2(\partial \psi) = \res_2(\inf_1 (\varphi))$
	\item $\inf_2 (\varphitilde)(\rho^{k_1} \sigma_1, \, \rho^{k_2} \sigma_2) = \frac{(\inf_1 \varphi)(\rho^{k_1}\sigma_1,\,\rho^{k_2} \sigma_2)}{(\partial \psi)(\rho^{k_1} \sigma_1,\, \rho^{k_2} \sigma_2)}$ is independent of the choice of elements $\sigma_1, \sigma_2 \in H_2$.
	\item $\res_2(\psi) = \deltatilde$, where $\partial \deltatilde = \res_2 (\inf_1 (\varphi))$
\end{enumerate}

It will be of use to observe that \hyperref[cocycle_conditions]{(3)} expands to:

\begin{equation} \label{eqn: cocyclecond3} 
	\frac{(\inf_1 \varphi)( \rho^{k_1}\sigma_1,\, \rho^{k_2}\sigma_2)}{(\partial \psi)(\rho^{k_1} \sigma_1,\, \rho^{k_2} \sigma_2)} = \frac{ (\inf_1 \varphi)(\rho^{k_1}\sigma_1,\, \rho^{k_2}\sigma_2) \cdot \psi(\rho_1^{k_1}\sigma_1 \rho_2^{k_2}\sigma_2)}{(\rho^{k_1} \sigma_1)\left(\psi(\rho^{k_2} \sigma_2)\right) \cdot \psi(\rho^{k_1} \sigma_1)}. \vspace{.5 em}
\end{equation}

Thus, in order to compute the local invariants, it suffices to determine $\psi(\rho)$. Indeed,
\begin{equation} \label{eqn: psirho} 
	\displaystyle \prod_{j=0}^{d-1} \varphitilde(\rho^j, \rho) 
	= \prod_{j=0}^{d-1} \textup{inf}_2 (\varphitilde)(\rho^j, \rho)
	= \prod_{j=0}^{d-1} \frac{ (\textup{inf}_1 \varphi)(\rho^j, \rho) \cdot \psi(\rho^{j+1})}{\rho^j \psi(\rho) \cdot \psi(\rho^j)}
	= \prod_{j=0}^{d-1} \frac{ (\textup{inf}_1 \varphi)(\rho^j, \rho)}{ \rho^j(\psi(\rho))}
\end{equation}

This will depend on the structure of the Galois group $\Gtilde$, and we will detail some of the relevant cases below. First, however, we need a more general approach for determining whether a given 1- or 2-cocycle is a coboundary, and producing the corresponding 0- or 1-cochain. We give a particular application to the case of dihedral groups following \cite{Preu}.

\subsection{Dihedral groups} \label{subsec: dihedral}
Suppose that $G \isom D_n = \langle g,h \mid g^n = h^2 = (gh)^2 = 1 \rangle$ is the Galois group of some finite extension of fields, $K/k$. Let $K^\times$ be the usual $G$-module. Two-cocycles for $H^2(G,K^\times)$ are represented, using the efficient resolution of \S\ref{subsec: efficres}, by triples $(r,s,t)$ of nonzero elements $r \in K^{\langle g \rangle}$, $s \in K^{\langle h \rangle}$, and $t \in K^{\langle gh \rangle}$, satisfying $Nr = Ns Nt$, where each $N$ denotes the norm from the respective field down to $k$. Coboundaries are triples of the form $(r,s,t) = (N_g r', N_h s', N_{gh}(\frac{r'}{s'}))$, for $r',s' \in K^\times$. To determine whether a given 2-cocycle is a 2-coboundary, one first needs to map back to the standard resolution to obtain a 2-cocycle $f: G \times G \to K^\times$. This can be accomplished using the morphisms of chain complexes given in Lemma \ref{lemma: quasi_isom}. In particular, one computes that $(r,s,t)$ maps to $f \in H^2(G,K^\times)$ for the standard resolution such that
\begin{eqnarray*}
f(g^i, g^{i'} h^{j'}) = \begin{cases} r^{-1}, & \textup{ if } i + i' > n \\
1, & \textup{ otherwise }
\end{cases} \\
f(g^ih, g^{i'} h^{j'}) = \begin{cases} \frac{r}{g^{i-i'} s^{j'} \prod_{e = 1}^{i'} (g^{i-e}.t(g.s)) }, & \textup{if } i' > i \bigskip \\ 
\frac{1}{g^{i-i'} s^{j'} \prod_{e = 1}^{i'} (g^{i-e}.t(g.s))}, & \textup{ otherwise} \end{cases}
\end{eqnarray*}

If $f$ is cohomologically trivial, then one can apply the methods of \cite[Prop 4.1]{Preu} to obtain the 1-cocycle lift $f' \in H^1(G,K^\times)$, 
\begin{equation} \label{eqn:1cocyclelift} f'(g^i h^j) = g^ih.(s')^{-j} \prod_{e = 0}^{i-1} g^e.r'^{-1} \end{equation}

\subsection{Direct Product of Groups} 
Let us now assume that $K_w/k_v$ is a dihedral Galois extension of degree $2n$, with Galois group $Q_1 = D_n$. Let $Q_2$ be the Galois group of the degree $d$ extension $\kcycl/k_v$, isomorphic to the cyclic group of order $d$. Suppose further that $\kcycl \cap K_w = k_v$ and that $\Gtilde \isom H_1 \times H_2$. The groups have the following presentations:
\begin{eqnarray*}
	H_1 & = & \langle a \mid a^d = 1 \rangle \isom Q_2\\
	H_2 & = & \langle b,c \mid b^2 = c^n = (cb)^2 = 1 \rangle \isom Q_1 \\
	\Gtilde  & = & \langle a,b,c \mid a^d = b^2 = c^n = (cb)^2 = [a,c] = [a,b] = 1 \rangle
\end{eqnarray*}
Our starting data will be a triple $(r,s,t) \in H^2(Q_1, K_w^\times)$, arising as the specialization of the $\calA \in H^2(G, \calO(U_K)^\times)$ to a point $P_v \in X(k_v)$. As above, this 2-cocycle representative for the efficient resolution can be pulled back to $\varphi \in H^2(Q_1,K_w^\times)$ for the standard resolution. Moreover, we observe that the composition $\res_2 \circ \inf_1$ applied to $(r,s,t)$ yields a 2-cocycle $H^2(H_2, A)$ which may also be written in terms of the efficient resolution as $(r,s,t)$, just now viewed as elements of $A$ instead of $K_w^\times$. However, by equation (\ref{eqn:infres}), this must be cohomologically trivial, hence by the discussion above, there exist $r', s' \in A$ such that $N_c(r') = r$ and $N_b(s') = s$. Furthermore, after mapping to the standard resolution, this is the lift of a 1-cocycle $\deltatilde \in H^1(H_2,A)$, defined by equation (\ref{eqn:1cocyclelift}), i.e. $\deltatilde(c^i b^j) = c^ib.(s')^{-j} \prod_{e = 0}^{i-1} c^e.r'^{-1} $.

In this case, since $\Gtilde$ is a direct product, $\inf_1 \varphi(a^{k_1} \sigma_1 , \, a^{k_2} \sigma_2) = \inf_1 \varphi( \sigma_1 , \, \sigma _2)$. Hence, together with property \hyperref[cocycle_conditions]{(2)}, equation (\ref{eqn: cocyclecond3}) further expands to:

\[ 
	\frac{(\inf_1 \varphi)( a^{k_1}\sigma_1,\, a^{k_2}\sigma_2)}{(\partial \psi)(a^{k_1} \sigma_1,\, a^{k_2} \sigma_2)} = \frac{ \sigma_1(\psi(\sigma_2))\cdot \psi(\sigma_1) \cdot \psi(a^{k_1 + k_2} \sigma_1 \sigma_2)}{ \psi(\sigma_1 \sigma_2) \cdot (a^{k_1} \sigma_1)\left(\psi(a^{k_2} \sigma_2)\right) \cdot \psi(a^{k_1} \sigma_1)} \vspace{.5 em}
\]

Taking, for example, $\sigma_1 = \sigma$, $\sigma_2 = \id$, $k_1 = 1$, $k_2 = 0$, we obtain:

\[ 
	\frac{\sigma (\psi(\id)) \cdot \psi(\sigma) \cdot \psi(a \sigma)}{\psi(\sigma)\cdot (a \sigma)(\psi(\id)) \cdot \psi(a \sigma)} = 1 \vspace{.5 em}
\]

Since this quantity is independent of the choice of $\sigma_1, \sigma_2$, letting $\sigma_1 = \id$ and $\sigma_2 = \sigma$ instead, we obtain:

\begin{equation} \label{eqn: psi_a_sigma_1} 
	1 = \frac{\psi(\sigma) \psi(a\sigma)}{\psi(\sigma) a(\psi(\sigma)) \psi(a)} = \frac{\psi(a \sigma)}{a(\psi(\sigma)) \cdot \psi(a)} \vspace{.5 em} 
\end{equation}

Similarly, taking $k_1 = 0$, $k_2 = 1$, $\sigma_1 = \id$, $\sigma_2 = \sigma$, equation (\ref{eqn: cocyclecond3}) simplifies to: 

\[ 
	\frac{\psi(\sigma) \cdot \psi(\id) \cdot \psi(a \sigma)}{\psi(\sigma) \cdot \psi(a\sigma) \cdot \psi(\id)} = 1 \vspace{.5 em}
\]

And, again, since this quantity is independent of $\sigma_1, \sigma_2$, choosing $\sigma_1 = \sigma$ and $\sigma_2 = \id$ instead yields: 

\begin{equation} \label{eqn: psi_a_sigma_2}
	1 = \frac{\psi(\sigma) \cdot \psi(a \sigma)}{\psi(\sigma) \cdot \sigma(\psi(a)) \cdot \psi(\sigma)} = \frac{\psi(a\sigma)}{\sigma(\psi(a))\cdot \psi(\sigma)} \vspace{.5 em}
\end{equation}

Together, equations (\ref{eqn: psi_a_sigma_1}) and (\ref{eqn: psi_a_sigma_2}) give:

\begin{equation} \label{eqn: psi_a_sigma} 
	\frac{\sigma(\psi(a))}{\psi(a)} = \frac{a\left(\psi(\sigma)\right)}{\psi(\sigma)},\, \hspace{.5 em} \textup{for each } \sigma \in H_2.
\end{equation}

Thus, $\psi(a)$ is determined by a Hilbert 90 condition, which can be solved explicitly. Indeed, by property \hyperref[cocycle_conditions]{(4)}, for $\sigma = c^i b^j$
\[ 
	\psi(\sigma) = \deltatilde(c^i b^j) = c^i.(s'^{-j}) \prod_{e =0}^{i-1} c^e.(r'^{-1}). 
\]

Moreover, in this case $(\inf_1 \varphi)(a^j,a) = 1$, hence by equation (\ref{eqn: psirho}), we find
\[ 
	\inv_v(\varphitilde) = -\frac{1}{d} v_d \left( \prod_{j=0}^{d-1} a^j (\psi(a)) \right).
\]

\subsection{Semidirect Products}
Our analysis in this case is similar. Suppose that the groups now have the following presentations: 
\begin{eqnarray*}
	\Gtilde & = & \langle b,c \mid c^n = b^d = 1 \mid bc = c^{-1} b \rangle \\
	H_1 & = &  \langle b^2 \rangle \\
	H_2  & = &  \langle c \rangle \\
	Q_1 & = & \langle g, h \mid g^n = h^2 = (gh)^2 = 1 \rangle,
\end{eqnarray*}
with $c \equiv g \mod H_1$ and $b \equiv h \mod H_1$.

As before, we begin with a 2-cocycle representative $(r,s,t) \in H^2(Q_1, K_v^\times)$. In this case, an application of $\res_2 \circ \inf_1$ yields a class in $H^2(H_2,A)$ which is represented by the single element $r \in A$, arising from the efficient resolution for cyclic groups. In particular, since this class is again cohomologically trivial, we obtain the 1-cocycle lift $\deltatilde: H_2 \to A$ such that
\[ \deltatilde(c^i) = \prod_{e = 0}^{i-1} c^e.r'^{-1},\]
where $r' \in A$ is such that $N_c(r') = r$. 

In this case, we shall be interested in determining $\psi(b)$. We shall need the following:
\begin{multicols}{2}
    \begin{itemize}
        \item $(\textup{inf}_1 \varphi)(1,b) =  \varphi(1,h)  = 1$
        \item $(\textup{inf}_1 \varphi)(b,1) =  \varphi(h,1)  = 1$
        \item $(\textup{inf}_1 \varphi)(b,b) =  \varphi(h,h)  = s^{-1}$
        \item $(\textup{inf}_1 \varphi)(c,b) = \varphi(g,h)  = 1$
        \item $(\textup{inf}_1 \varphi)(cb,b)=  \varphi(gh,h) = c(s)^{-1} $
        \item $(\textup{inf}_1 \varphi)(cb,c) =  \varphi(1,h) = (t c(s))^{-1}$
        \item $\psi(c^j) = \deltatilde(c^j) = \prod_{e=0}^{j-1} c^e.(r'^{-1})$
    \end{itemize}
\end{multicols}


Then equation (\ref{eqn: cocyclecond3}) yields:
\[ 
	\frac{(\textup{inf}_1 \varphi)(cb,c)}{(\partial \psi)(cb,c)} = \frac{(\textup{inf}_1 \varphi)(b,1)}{(\partial \psi)(b,1)} = 1 \implies (\textup{inf}_1 \varphi)(cb,c) =  \frac{c(\psi(cb)) \psi(c)}{\psi(b)}.
\]

Similarly,
\[ 
	\frac{(\textup{inf}_1 \varphi)(c,b)}{(\partial \psi)(c,b)} = \frac{(\textup{inf}_1 \varphi)(1,b)}{(\partial \psi)(1,b)} = 1 \vspace{.5 em}
\]

And again, substitution and simplification gives:
\[ 
	\psi(cb) = b(\psi(c)) \cdot \psi(b). 
\]

Thus, we obtain
\[ 
	(\textup{inf}_1 \varphi)(cb,c) = \frac{cb(\psi(c)) \cdot \psi(c) \cdot c(\psi(b))}{\psi(b)}. 
\]

Hence,
\begin{equation} 
	\frac{c (\psi(b))}{\psi(b)} = \frac{r' cb(r')}{t c(s)},
\end{equation}
and indeed, the right hand side of the equation is norm 1 under $c$, so $\psi(b)$ can be found via an application of effective Hilbert 90.

Finally, in this case, $(\inf_1 \varphi)(b^{2\ell+1},b) = s^{-1}$, and $(\inf_1 \varphi)(b^{2\ell},b) = 1$, for $\ell \in \Z$. Thus, the invariant is given by equation (\ref{eqn: psirho}), that is
\[ 
	\inv_v [\varphitilde] = \frac{1}{d} v_d \left( \prod_{j=0}^{d-1} \frac{s^\frac{d}{2}}{b^j(\psi(b))} \right).
\]

\subsection{Computing Local Invariants}
We return briefly to the setting when $\Gtilde$ is a direct product of groups.
\subsubsection{Calibrating $r'$ and $s'$}
Given a triple $(r,s,t) \in (A^\times)^3$ representing a 2-coboundary in $H^2(D_n, A^\times)$ such that $N_b(r) = N_c(st)$,  $N_c(r') = r$, and $N_b(s') = s$, it may be possible that the choice of $r'$ and $s'$ must be modified by constants so that $N_{cb}\left (\frac{r'}{s'} \right) = t$. We describe here how to find such constants. Let $\lambda := \frac{t}{N_{cb}\left(\frac{r'}{s'} \right)}.$
It suffices to find a $\chi := \frac{\chi_r}{\chi_s} \in A^\times$ such that
\vspace{-1 em}
\begin{multicols}{3}
	\begin{enumerate}
		\item $r'' := \displaystyle \frac{r' \cdot c(\chi_r)}{\chi_r}$ \vspace{.5 em}
		\item $s'' := \displaystyle \frac{s' \cdot b(\chi_s)}{(\chi_s)}$ \vspace{.5 em}
		\item $ N_{cb}\left( \frac{\frac{c(\chi_r)}{\chi_r}}{\frac{b(\chi_s)}{\chi_s}} \right) = \lambda$  
	\end{enumerate}
\end{multicols}

The last condition is equivalent to
\[ 
	\frac{c \left(\frac{b(\chi)}{\chi}\right)}{\left(\frac{b(\chi)}{\chi}\right)} = \lambda.
\]
Thus, we first solve for $u_\lambda$ such that $\frac{c(u_\lambda)}{u_\lambda} = \lambda$. This is done via effective Hilbert 90 for the extension $A/A^{\langle c \rangle}$. That is,
\[ 
	u_\lambda^{-1} = 1 + \lambda + \lambda \cdot c(\lambda) + \dots + \prod_{j=0}^{n-1} c^j(\lambda)
\]

If $N_b(u_\lambda) = 1$, then we are done, via another appliction of effective Hilbert 90 for the extension $A/A^{\langle b \rangle}$. Otherwise, we must find $\nu_\lambda$ such that $\frac{c(\nu_\lambda)}{\nu_\lambda} = 1$ (i.e. $\nu_\lambda \in A^{\langle c \rangle}$) and $N_b(\nu_\lambda) N_b(u_\lambda) = 1$. Then, 
\[ 
	\frac{c\left(u_\lambda \nu_\lambda\right)}{u_\lambda \nu_\lambda} = \lambda.
\]
And since $N_b(u_\lambda \nu_\lambda) = 1$, $\chi$ is found via effective Hilbert 90:
\[ 
	\chi^{-1} := 1 + u_\lambda \nu_\lambda,
\]
so that $\frac{b(\chi)}{\chi} = u_\lambda \nu_\lambda$, as desired. 

\subsubsection{Computing $\psi(a)$} 

Finally, $\psi(a)$ must satisfy equation \ref{eqn: psi_a_sigma} for each $\sigma \in H_2$, hence in particular for the generators $b,c$. Thus, again, we employ an iterative effective Hilbert 90 as follows. Let $u_r \in A^\times$ be such that $c(u_r)/u_r = a(r')/r' = c(\psi(a))/\psi(a)$. To find $\psi(a)$ it suffices to find an element $\mu$ fixed by c, such that $b(\mu u_r)/(\mu u_r) = a(s')/s' = b(\psi(a))/\psi(a)$. Then,
\[
	\frac{b(\mu)}{\mu} = \frac{a(s')b(u_r)}{s'u_r},
\]
hence, taking $\mu$ to be $\mu:= 1 + \frac{a(s')b(u_r)}{s'u_r}$ suffices, and we find that $\psi(a) = u_r \mu$. The cocycle lifting algorithm has been implemented in $\magma$.

\subsection{Modifying by a base point}
As noted in \cite{Bright}, an oft used strategy for proving that a class $\calA \in \Br X$ of order $n$ gives no obstruction to the Hasse principle is to demonstrate the existence of a finite place $v$ of $k$ such that the evaluation map $X(k_v) \to (\Br k_v)[n]$, sending a point $ P \in Y(k_v)$ to the evaluation $\calA(P) \in (\Br k_v)[n]$, is surjective. In that case, any adelic point in $X(\A_k)$ can be modified at the place $v$ to produce an adelic point orthogonal to $\calA$.

This strategy is complicated by the fact that surjectivity can be destroyed by simply changing $\calA$ by a constant algebra in $\Br k$. Adding a constant algebra of arbitrarily large order can increase the order of $\calA$, yet the image of the evaluation map is translated hence stays the same size. Thus, to compensate for this deficiency, one can fix a base point $P \in X(k_v)$, and instead consider the modified evaluation map $X(k_v) \to \Br k_v$ defined by $P' \mapsto \calA(P') - \calA(P)$. This evaluation map only depends on the class of $\calA$ modulo constant algebras. If the class of $\calA$ in $\Br X / \Br k$ has order $n$ and the evaluation map surjects onto $(\Br k_v)[n]$, then $\calA$ gives no obstruction to the Hasse principle.

In the next section we will produce examples of surfaces which fail to have a Brauer-Manin obstruction due to surjectivity of an invariant at a single place. The relevant Brauer classes are defined only up to a constant algebra which is not easily computed, and thus modification by a base point is a useful tool in practice.


\section{Examples} \label{sec: concrete examples}

\subsection{Insufficiency of the integral Brauer Manin Obstruction}
Finally, we return to the case of affine Ch\^atelet surfaces $X_m/\Q: x^2 + y^2 + t^4 = m$.

\begin{theorem}{\cite[Thm 8.1]{CTSSD87b}} \label{thm: CTSSD87 HP}
	Let $k$ be a number field, and let $Y$ be a smooth proper model of a Ch\^atelet surface (\ref{eqn: Xoverk}), then the Brauer-Manin obstruction is the only one to the Hasse principle on $Y$. Moreover, if $P(t)$ is an irreducible polynomial (of degree 4), the Hasse principle holds for $Y$.
\end{theorem}

\begin{cor}{\cite[Cor. 8.12]{CTSSD87b}} \label{cor: CTSSD87 Xm}
	Let $m$ be a natural number. Write $m = 2^{4r +i}n$ with $r$ integral, $i$ integral and $0 \le i \le 3$, and $n$ an odd integer. Then $m$ can be written $m = u^2 + v^2 + w^4$ with $u,v,w \in \Q$ if and only if $m$ does not fall in one of the following cases:
	\begin{enumerate}[(i)]
		\item $i = 0$, and $n \equiv 7 \pmod{8}$
		\item $i =2$, and $n \equiv 3 \pmod{4}$
	\end{enumerate}
\end{cor} 

That is, the smooth proper compactifications of the varieties $X_m$ satisfy the Hasse principle, hence $X_m(\Q) \ne \emptyset$ whenever $X_m$ is everywhere locally soluble. However, there exist $m$ for which no such representation into the sum of two squares and a fourth power exists over the integers, despite the existence of $\Z_p$ points for each prime $p$. The first few values of such $m$ are $22,43,67,70,78,93,177$. As remarked in \cite{CTSSD87b}, it is a conjecture in additive number theory that any integer large enough should admit such a representation in the integers as soon as it is everywhere locally soluble. 

\begin{prop} \label{lem: insuff} For each $m$ listed above, there exists a prime $p|m$ such that the local invariant map $\inv_p: \calX(\Z_p) \to \frac{1}{4}\Z/\Z$ is surjective, hence the Brauer-Manin obstruction is insufficient to explain the failure of existence of integral points for these varieties.
\end{prop}

\begin{proof} By \cite[Rem 8.12.3]{CTSSD87b}, the varieties $X_m$ are soluble in $\Z_p$ for all $p$ odd, and are soluble in $\Z_2$ iff $m$ is not of the form given in Cor \ref{cor: CTSSD87 Xm} and $m$ is not divisible by 4. We focus on the particular case of $m = 22$. The reader should consult \hyperref[fig1invts]{Figure 1} for the remaining data. Here $K = \Q(i, \sqrt[4]{22})$. By Corollary \ref{cor: Koch}, since $\Q_2(i,\sqrt[4]{22})$ is a totally ramified extension of $\Q_2$ of degree $8 = [K:\Q]$, so $H^3(G, K^\times) = 0$. Thus the conditions of Proposition \ref{prop: KT08 II} are satisfied. By Theorem \ref{thm: BrX_explicitrep}, the restriction to $U$ of a generator of $\Br X_{22}/\Br_0 X_{22}$ is represented explicitly by
\[ \calA:= \left( -44(x + iy), \sqrt[4]{22}(t - \sqrt[4]{22}), (1 +i)\sqrt[4]{22} \right) \in H^2(G, \calO(U_K)^\times)\]
This is a lift of the class $\Btilde$ which generates $H^1(G, \Pic (X_{22})_K)$, and thus it agrees with $\Btilde$ up to a constant algebra in $\Br \Q$. 
Over a field $F$ we recall from \S\ref{subsec: dihedral} that a triple $(R,S,T) \in F^{\times 3}$ defines a 2-cocycle in $H^2(G,F^\times)$ if \[  R \in F^{\langle g \rangle}, S \in F^{\langle h \rangle},  \textup{ and } T \in F^{\langle gh \rangle}.\] The 2-coboundaries are those triples $(R,S,T) = \left( N_g(R'), N_h(S'), N_{gh}(\frac{R'}{S'})\right)$ for some elements $R',S' \in F^\times$. For ease of computation, we work with twice this class (under coordinate-wise multiplication), with representative
	\[ (R,S,T) = \left( (x + iy)^2, (t - \alpha)(t+\alpha), 1 \right).\]
Then one can check that this is a nontrivial class as 
	$N_h((x+iy)^2) = (x^2 +y^2)^2 = N_g( (t-\alpha)(t + \alpha)) = P(t)^2$, and $(x+iy)^2$ is not a norm. Doubling again yields
\[ \left( (x + iy)^4, (t-\alpha)^2 (t + \alpha)^2 , 1 \right) = (N_g(R'), N_h(S'), 1),\]
with $R' = (x + iy)$ and $S' = (t-\alpha)(t+\alpha)$. However,
\[ N_{gh}\left(\frac{R'}{S'} \right)= \frac{(x+iy) (x -iy)}{(t - \alpha)(t+\alpha)(t-i\alpha)(t+i\alpha)} = \frac{(x^2 + y^2)}{P(t)} = -1 \]
This demonstrates that $\calA$ is of order 8, as $(1,1,-1)$ is not a 2-coboundary (no $ R'', S'' \in \calO(U_K)^\times$ exist such that $N_g(R'') = 1$, $N_h(S'') = 1$ and $N_{gh}\left( \frac{R''}{S''} \right) = -1$), but $(1,1,1)$ is. Thus when computing local invariants after specialization, we must carry out cocycle lifting over the unramified cyclotomic extension of $\Q_v$ of degree $d =8$.

We note, however, that the class $\calA$ does have the correct order in $\Br U/\Br k$; that is, we are in the setting of Claim \ref{claim: splitting}, in which there exists a splitting. To see this, we must check that $ - 1 N_g(\rho) = N_h(\xi)$ has solutions,
	where $\rho \in K^{\langle h \rangle}$ and $\xi \in K^{\langle g \rangle}$. We shall consider both $\rho, \xi$ as elements of $K$, and rearrange to write the constant $-1$ in terms of a single norm:
	{\small
	\begin{eqnarray*}
	-1 & = & \frac{N_h(\xi)}{N_g(\rho)} \\
	 & = & \frac{\xi h(\xi)}{\rho g(\rho) g^2(\rho) g^3(\rho)} \\
	 & = & \frac{g (\xi) g h g (\xi) }{\rho g(h\rho) g^2(\rho) g^3(h\rho)} \\
	 & = & \frac{N_{gh}( g \xi)}{\rho g(h\rho) g^2(\rho) gh(g^2\rho)} \\
	 & = & N_{gh}\left(\frac{\xi g(\xi)}{\rho g^2(\rho)}\right) \
	\end{eqnarray*}}
	Thus, the existence of a splitting for the variety $X_{22}$ (and more generally for $X_m$ with $m$ not a square or fourth power in $\Q$) is equivalent to the condition that $-1$ is a norm for the extension $K/K^{\langle gh \rangle}$. That is, whether the restriction of the quaternion algebra $(-1,-1) \in \Br K$ to the field $ \Q(\sqrt[4]{22}(1 + i))$ is trivial. But this extension has no real places, so the archimedean invariants are all 0.  Additionally, 2 is totally ramified in this extension, and as $\Res$ acts as multiplication by the degree for local fields, the local invariants are all trivial.

When $v = 2$, we are thus in the setting of \S4.1. After modifying by a choice of base point $P = (1,10,\sqrt[4]{-79}) \in \Z_2^3$ and applying the algorithm of $\S4$ we find the surjectivity of the map $\inv_2: \calX(\Z_2) \to \frac{1}{4}\Z/\Z$, and hence a non-empty Brauer-Manin set.

\begin{figure} \small \label{fig1invts}
\caption{Local invariant computations for $\calX_m$}
\begin{align*}
\begin{tabular}{|c|c|c|c|}
\hline
\hspace{.2 em }$m$ \hspace{.2 em} & \hspace{.2 em} $p$  \hspace{.2 em} & ${(x,y)}$ & invariant \\
\hline \hline
{} & {} & (1,10) & 0   \\
22 & 2  & (2,15) & 1/4 \\
{} & {} & (2,1)  & 1/2 \\
{} & {} & (1,6)  & 3/4 \\
\hline
{} & {} & (1,17) & 0   \\
43 & 43 & (1,35) & 1/4 \\
{} & {} & (1,8)  & 1/2 \\
{} & {} & (2,2)  & 3/4 \\
\hline
{} & {} & (1,2)  & 0   \\
67 & 67 & (11,2) & 1/4 \\
{} & {} & (2,4)  & 1/2 \\
{} & {} & (2,3)  & 3/4 \\
\hline
\end{tabular} & \hspace{3 em}
\begin{tabular}{|c|c|c|c|}
\hline
\hspace{.2 em }$m$ \hspace{.2 em} & \hspace{.2 em} $p$  \hspace{.2 em} & ${(x,y)}$ & invariant \\
\hline \hline
{} & {} & (1,2)  & 0   \\
70 & 2  & (1,6)  & 1/4 \\
{} & {} & (1,10) & 1/2 \\
{} & {} & (1,14) & 3/4 \\
\hline
{} & {} & (2,11) & 0   \\
78 & 2  & (2,13) & 1/4 \\
{} & {} & (2,3)  & 1/2 \\
{} & {} & (2,5)  & 3/4 \\
\hline
{} & {} & (1,17) & 0   \\
93 & 31 & (1,5)  & 1/4 \\
{} & {} & (1,4)  & 1/2 \\
{} & {} & (1,8)  & 3/4 \\
\hline
\end{tabular}
\end{align*}
\end{figure}
\end{proof}

As noted in the introduction, the compactness of $X_m(\R)$ forces the possible set of integral points to be finite. However, one may remove this compactness constraint by instead asking for the existence of $\Z[\frac{1}{2}]$ points. Again we find,

\begin{cor} 
	Let $S = \{2, \infty\}$. The Brauer-Manin obstruction is insufficient to explain failures of the S-integral Hasse principle for the varieties $X_m$.
\end{cor}

\begin{proof} 
	 Suppose that $\left(\frac{x_0}{2^{r_0}}, \frac{y_0}{2^{r_1}}, \frac{t_0}{2^{r_2}}\right) \in X_m(\Z[\frac{1}{2}])$. After clearing denominators, the solubility is implied by the existence of solutions in the integers of
	\[	
		x^2 + y^2 + t^4 = 2^{4r_2}m 
	\]
	However, either $r_2$ is 0, and this is equivalent to the solubility in the integers of the original equation $X_m$ or $r_2 >0$ and the right hand side is divisible by 4, hence not everywhere locally soluble, a contradiction. Thus if $m$ is one of the values above, $X_m$ is not soluble over $\Z[\frac{1}{2}]$, and Proposition \ref{lem: insuff} shows that there is still no Brauer-Manin obstruction.
\end{proof}


\end{document}